\documentclass[a4paper,11pt]{amsart}
\usepackage[colorlinks,citecolor=red,urlcolor=blue,bookmarks=false,hypertexnames=true]{hyperref} 
\usepackage{amsmath, amssymb, amsfonts, amsbsy, amsthm, latexsym, epsfig}
\usepackage{epstopdf }
\usepackage{graphicx}
\usepackage[shortlabels]{enumitem}
\usepackage{color}
\usepackage{lmodern}
\definecolor{applegreen}{rgb}{0.55, 0.71, 0.0}
\definecolor{cadmiumgreen}{rgb}{0.0, 0.42, 0.24}
\definecolor{burntorange}{rgb}{0.8, 0.33, 0.0}

\theoremstyle{plain}
\newtheorem{theorem}{Theorem} [section]
\newtheorem{lemma}[theorem]{Lemma}
\newtheorem{proposition}[theorem]{Proposition}

\theoremstyle{definition}

\newtheorem*{teo*}{Theorem}

\theoremstyle{definition}
\newtheorem{definition}[theorem]{Definition}

\newtheorem{remark}[theorem]{Remark}

\newtheorem{example}[theorem]{Example}

\newcommand{\esssup}{{\mathrm{ess}\sup}}

\newcommand{\cT}{\mathcal{T}}

\newcommand{\R}{\mathbb{R}}
\newcommand{\Z}{\mathbb{Z}}
\newcommand{\N}{\mathbb{N}}
\newcommand{\CC}{\mathbb{C}}

\newcommand {\w} {\omega}
\newcommand{\bigpoplus}{\overset{.}{\bigoplus}}
\newcommand{\poplus}{\overset{.}{\oplus}}

\newcommand{\supp}{{\rm supp\,}}
\newcommand{\rank}{{\rm rank\,}}

\begin{document}

\title{Diagonalization of Shift-preserving Operators}

\let\thefootnote\relax\footnote{


The research of the authors is partially supported by grants: UBACyT 20020170100430BA, PICT 2014-1480 (ANPCyT) and CONICET PIP 11220150100355. In particular VP is also supported  by UBACyT 20020170200057BA and PICT-2016- 2616.

}

\author{A. Aguilera}
\address{ Departamento de Matem\'atica, Universidad de Buenos Aires,
	Instituto de Matem\'atica "Luis Santal\'o" (IMAS-CONICET-UBA), Buenos Aires, Argentina}
\email{aaguilera@dm.uba.ar}

\author{C. Cabrelli}
\address{ Departamento de Matem\'atica, Universidad de Buenos Aires,
	Instituto de Matem\'atica "Luis Santal\'o" (IMAS-CONICET-UBA), Buenos Aires, Argentina}
\email{cabrelli@dm.uba.ar}

\author{D. Carbajal}
\address{ Departamento de Matem\'atica, Universidad de Buenos Aires,
	Instituto de Matem\'atica "Luis Santal\'o" (IMAS-CONICET-UBA), Buenos Aires, Argentina}
\email{dcarbajal@dm.uba.ar}

\author{V. Paternostro}
\address{ Departamento de Matem\'atica, Universidad de Buenos Aires,
	Instituto de Matem\'atica "Luis Santal\'o" (IMAS-CONICET-UBA), Buenos Aires, Argentina}
\email{vpater@dm.uba.ar}

\begin{abstract}
In this note we study the structure of shift-preserving operators 
ac\-ting on a finitely generated shift-invariant space. 
We define a new notion of diagonalization for these operators, 
which we call $s$-diagonalization. We give necessary and sufficient  conditions on a bounded shift-preserving operator 
in order to be $s$-diagonalizable. These conditions are  in terms of its range operator. We also
obtain a generalized Spectral Theorem for normal bounded shift-preserving operators.

\end{abstract}

\maketitle
\section{Introduction}
Shift-invariant spaces are subspaces of $L^2(\R^d)$ that are invariant under the action of translations by integer vectors.
These spaces have been used in approximation theory, sampling theory, and wavelets, and their structure is very well known.
See \cite{B,BDR1,BDR2,H, RS} in the euclidean case, and \cite{BHP15, BR14,CP10} in the context of topological groups.

Given an (at most countable) set of functions $\Phi \subset L^2(\R^d)$ the subspace
$$
S(\Phi) :=  \overline{\text{span}} \left\{ T_k\varphi\,:\,\varphi \in \Phi, \,k\in\mathbb Z^d \right\}
 $$
 is  shift invariant  and $\Phi$ is called a {\it set of generators}. Moreover, every shift-invariant space  is of this form. A {\it finitely generated} shift-invariant space is one that has a finite set of generators and the {\it length} is the minimum cardinal between all sets of generators.
 
Each shift-invariant space $V$ has an associated {\it range function} (see Def.~\ref{range-function}) which  represents $V$  as a measurable field of closed subspaces (the fiber spaces) of $\ell^2(\Z^d).$ The  functions of $V$ have its fibers in the fiber space (see Section \ref{section-SIS}) and the connection between $V$ and its associated range function is through an isometric isomorphism that we denote by $\cT.$

The representation of a shift-invariant space through its fiber spaces is a key to studying its structure.
In fact, the   integer translates of a set of  functions in the space form  a basis or a frame if and only if  the fibers of such functions form a basis or a frame of the fiber spaces with uniform bounds (see \cite{B}). On the other hand, when the shift-invariant space is finitely generated, all these fiber spaces are finite-dimensional. As a consequence, the representation by fiber spaces allows us to translate problems that involve  shift-invariant spaces into linear algebra. 
 
 The natural operators acting on these spaces are those that commute with integer translates, i.e {\it shift-preserving} operators.
In \cite{B}, Bownik considers shift-preserving operators acting on  shift-invariant spaces and studies its properties through the concept of {\it range operator} (see Def. \ref{range operator}).
A  range operator is a representation of a shift-preserving operator as a measurable field of linear operators, obtained through the intertwining map $\cT$ mentioned before.
Each fiber of a range operator acts on the respective fiber space of the range function.
In this way, the action of the range operator through the fiber spaces permits us to decode the behavior of the associated shift-preserving operator (see Section \ref{section-SP}).
For instance, Bownik proved in \cite{B} that certain properties such as normality and unitaryness of a shift-preserving operator are inherited by its fibers.

When a shift-preserving operator acts on a  finitely generated shift-invariant space, each fiber of the  range operator is a  linear transformation acting on a finite dimensional subspace of $\ell^2(\Z^d).$ So, considering appropriate bases of the fiber spaces, one can represent the range operator
as a field of matrices. Given that, we want to translate the structure of these matrices 
(or the linear transformations that they induce) back to the shift-preserving operator. 

There are two delicate issues  when one tries to pursue this program.
 First, we  need some consistency in the structure of the field of  linear transformations. Here, the difficulty is to understand which is the  precise uniformity 
 condition to ask, to be able to obtain a similar property for the shift-preserving operator.
 The second issue is measurability. Since the field is measurable,  the objects associated with our linear transformations, such as eigenvalues, eigenvectors and  kernels, have to be measurable respect to  its behavior along the field. This requires a significative effort.

In this paper, we focus on a special notion of diagonalization for a shift-preserving operator acting on a finitely generated shift-invariant space, which we call {\it $s$-diagona\-lization} (see Def.~\ref{s-diag}). This notion generalizes the usual concept of diagonalization.
An $s$-diagonalizable shift-preserving operator is associated with a decomposition of the underlying shift-invariant space into a finite direct sum of closed subspaces, each of them being  invariant under the shift-preserving operator and under  integer translates.  The shift-preserving operator has a simple form when it acts in each of the invariant subspaces and when it is $s$-diagonalizable it can be written as a sum of simpler operators.

Towards the concept of  $s$-diagonalization, we define the notion of {\it $s$-eigenvalue}, a concept that generalizes the classical definition.
Let $L :V \rightarrow V$ be a shift-preserving operator acting on a shift-invariant space $V$ and $a=\{a_k\}_{k\in\Z^d} \in \ell^2(\Z^d)$  a 
 sequence of {\it bounded spectrum} (i.e. $\hat{a} \in L^{\infty}([0,1)^d)).$ Define $\Lambda_a := \sum_{k\in\Z^d} a_k T_k$ where
  $T_k$ is the translation operator by $k$.
  We say that $\Lambda_a$ is an $s$-eigenvalue of $L$ if  $\ker(L-\Lambda_a) \neq \{0\}.$
An $s$-eigenvalue is a very simple shift-preserving operator that acts as a convolution in the following sense:
Assume that $\{T_k \varphi_i : i=1,\dots,n, k\in \Z^d \}$ is a frame of $V$. 
If  $f =\sum_{i=1}^n \sum_{k\in \Z^d} b_i(k) T_k\varphi_i$ with $ b_1,\dots,b_n \in \ell^2(\Z^d),$ then
$\Lambda_a f =\sum_{i=1}^n \sum_{k\in \Z^d} (a*b_i)(k) T_k\varphi_i.$

One of our main results, states that a shift-preserving operator acting on a finitely generated shift-invariant space is $s$-diagonalizable if and only if its associated range operator is a field of diagonalizable matrices satisfying a uniformity angle condition.  

 As a consequence we obtain a type of spectral theorem for the case when the shift-preserving operator is  bounded and normal:

\begin{theorem}\label{spectral-intro}
Let $L :V\rightarrow V$ be a normal bounded shift-preserving operator on a finitely generated shift-invariant space $V.$
Then there exist $m\in \N$,  $s$-eigenvalues  $\Lambda_1,\dots \Lambda_m$ and orthogonal shift-invariant subspaces of $V$, $V_1,\dots, V_m$ 
such that 
$$ V=V_1\poplus\dots\poplus V_m\quad \text{ and }\quad L=\sum_{j=1}^m \Lambda_j P_{V_j},$$
where $P_S$ denotes the orthogonal projection onto a closed subspace $S.$
Furthermore, the shift-invariant subspaces $V_j$ are invariant under $L$ which implies that $Lf=\Lambda_j f$ for $f\in V_j.$
\end{theorem}

On the way, we prove that an invertible bounded shift-preserving  operator  is $s$-diagonalizable if and only if its inverse is $s$-diagonalizable,
and the decomposition of the inverse is obtained inverting the $s$-eigenvalues (see Proposition \ref{prop:inverso}). Using that a matrix with measurable entries has measurable eigenvalues, we explicitly construct $s$-eigenvalues for those shift-preserving operators acting on finitely generated shift-invariant spaces. Finally, we see that the minimum number of $s$-eigenvalues in a decomposition of a shift-preserving operator agrees with the essential supremum of the number of eigenvalues of the matrices of the associated field of operators.   

The paper is organized as follows. In Section \ref{section-SIS}, we give the basic known results about shift-invariant spaces and  range functions.
Section \ref{section-SP} is devoted to  shift-preserving operators. We  recall their definition, some known properties and prove (under a mild condition) that they can be represented  as matrices with measurable entries. Additionally, we  characterize the invertibility of a shift-preserving operator in terms of the invertibility  of its operator fibers. 
We introduce the concepts of $s$-eigenvalue and $s$-eigenspace for  a shift-preserving operator in Section \ref{section-s-eigenvalues} and study their connection with the eigenvalues and eigenspaces of  the associated range operator. Measurability issues related to $s$-eigenvalues and $s$-eigenspaces  are developed in Section \ref{section-measurability}.  Moreover, towards the existence  of $s$-eigenvalues of a shift-preserving operator, we construct measurable functions in Theorem \ref{thm:autovalores-full}  by pasting eigenvalues of the associated range operator.   Finally, in Section \ref{section-s-todo}, we introduce the notion of  $s$-diagonalization, show some of its properties and state and prove
our main result, that relates $s$-diagonalization with the diagonalization of the associated field of operators,  obtaining as a consequence Theorem \ref{spectral-intro}. 
 
\section{Shift-invariant Spaces}\label{section-SIS}

In this section we collect some of the properties of the theory of shift-invariant spaces that we will need later.

Here and for the remainder of this paper, the symbols $\poplus$ and $\oplus$ will denote the orthogonal sum and the direct sum of subspaces respectively.

\begin{definition}
	We say that a closed subspace $V\subset L^2(\mathbb R^d)$ is {\it shift invariant} if for each $f\in V$ we have that $T_k f\in V$, $\forall\,k \in \mathbb Z^d$. Here, $T_kf(x)= f(x-k)$.
\end{definition}

Given a countable set of functions $\Phi\subset L^2(\mathbb R^d)$, we will denote 
$$S(\Phi) = \overline{\text{span}} \left\{ T_k\varphi\,:\,\varphi \in \Phi, \,k\in\mathbb Z^d \right\}.$$
We say $\Phi$ is a set of generators of $V$ if $V = S(\Phi)$. When $\Phi$ is a finite set, we say that $V$ is a {\it finitely generated} shift-invariant space.
Moreover, we will denote
$$E(\Phi) = \left\{T_k\varphi\,:\,\varphi \in \Phi, \,k\in\mathbb Z^d \right\}.$$

An essential tool in the development of shift-invariant spaces theory is the technique known as fiberization that we  present now. We follow the notation used in \cite{B}.

\begin{proposition}\label{isometria}
	The map $\mathcal T: L^2(\mathbb R^d) \to L^2\left([0,1)^d,\ell^2(\mathbb Z^d)\right)$ defined by
	$$
	\mathcal T f(\omega) = \{\hat{f} (\omega+k) \}_{k \in \mathbb Z^d},
	$$
	is an isometric isomorphism. We call $\mathcal T f(\omega)$ the fiber of $f$ at $\omega$. Moreover, it satisfies that \begin{equation}\label{modulation}
	\mathcal{T} T_{k}f(\omega) = e_k(\omega) \,\mathcal T f(\omega),
	\end{equation}   
	where $e_k(\omega)=e^{-2\pi i\langle \omega,k\rangle}$. 
\end{proposition}

Here, the Hilbert space $L^2\left([0,1)^d,\ell^2(\mathbb Z^d)\right)$ consists of all vector-valued measurable functions $\psi: [0,1)^d\rightarrow \ell^2(\mathbb Z^d)$ with finite norm, where the norm is given by
 $$\|\psi\|=\left( \int_{[0,1)^d} \|\psi(\omega)\|^2_{\ell^2} \,d\omega\right)^{1/2}.$$

\begin{definition}\label{range-function}
A {\it range function} is a mapping
\begin{align*}
J:[0,1)^d&\rightarrow\{\text{\,closed subspaces of }\ell^2(\mathbb Z^d)\,\}\\
\omega&\mapsto J(\omega).
\end{align*}
\end{definition}

We say $J$ is measurable if the scalar function $\omega\mapsto\langle P_{J(\omega)}u,v \rangle$ is measurable for every $u,v\in\ell^2(\mathbb Z^d)$, where $P_{J(\omega)}$ is the orthogonal projection of $\ell^2(\mathbb Z^d)$ onto $J(\omega)$. Given a range function $J$, the following space can be defined
$$M_{J} = \{\psi\in L^{2}\left([0,1)^{d},\ell^{2}(\mathbb{Z}^{d})\right) : \psi(\omega)\in J(\omega),\text{ for a.e. }\w\in [0,1)^d\},$$
which is a closed {\it modulation-invariant} subspace of $L^2\left([0,1)^d,\ell^2(\Z^d)\right)$, i.e. for every $\psi \in M_J$ we have that $e_{k}\psi \in M_J$ for all $k\in \mathbb{Z}^{d}$. 

In \cite{B} and \cite{H}, it was proved that for every shift-invariant space $V\subset L^2(\mathbb R^d)$ there exists 
a measurable range function $J_V$ which sa\-tis\-fies that 
\begin{equation}\label{range function}
f\in V \;\text{if and only if }\mathcal T f(\omega)\in J_V(\omega), \: \text{for a.e. }\omega\in [0,1)^d,
\end{equation}
that is $\mathcal T (V) = M_{J_V}$.
On the other hand, every measurable range function $J$ defines a shift-invariant space, namely $V:= \mathcal T^{-1}(M_{J})$. When identifying range functions a.e. $\w\in [0,1)^d$, the correspondence between measurable range functions and shift-invariant spaces is one-to-one (see \cite[Proposition 1.5]{B}). 

Furthermore, if $V=S(\Phi)$ for some coun\-table set $\Phi\subset L^2(\mathbb R^d)$, then for a.e. $\omega\in[0,1)^d$,
\begin{equation*}\label{range-function-generated}
J_V(\omega) = \overline{\text{span}}\{\mathcal T f(\omega)\,:\,f\in\Phi\,\}.
\end{equation*} 
In particular, when $\Phi$ is a finite set, this allows us to translate problems in infinite dimensional shift-invariant spaces, into problems of finite dimension that can be treated with linear algebra. 

We call the subspace $J_V(\omega)$ the fiber space of $V$ at $\omega$ and we will simply denote $J$ when it is evident that we are referring to the corresponding range function of $V$. 

We have the following useful property.

\begin{lemma}\label{Helson projections}(Helson, 
\cite{H})
Let $V\subset L^{2}(\mathbb{R}^d)$ be a shift-invariant space with associated range function $J$. Then, for each $f\in L^2(\mathbb R^d)$ we have that
$$\mathcal T (P_V f)(\omega) = P_{J (\omega)}(\mathcal T f (\omega)),\quad\text{for a.e. }\w\in[0,1)^d.$$
\end{lemma}

\begin{definition}
The {\it length} of a finitely generated shift-invariant space $V\subset L^{2}(\mathbb{R}^{d})$ is denoted by $\mathcal{L}(V)$ and defined as the smallest natural number $\ell$ such that there exist $\varphi_{1},...,\varphi_{\ell} \in V$ with $V=S(\varphi_{1},...,\varphi_{\ell})$.
It is possible to give an equivalent definition of $\mathcal{L}(V)$ in terms of the range function $J$ associated to $V$, which is $$\mathcal{L}(V) = \underset{\omega\in [0,1)^d}{\text{\rm ess sup }}  \dim J(\omega).$$
 Since the range function of $V$ can take  the subspace $\{0\}$ as value, the {\it spectrum} of $V$ is defined by $$ \sigma(V) = \left\{\omega \in [0,1)^d\,:\,  J(\omega) \neq\{0\} \right\}.$$
\end{definition}

We now  summarize some properties of shift-invariant spaces that we will need along the paper. 
We refer for details and proofs to \cite{BDR2}.

\begin{lemma}\label{spec principal}
Let $V$ be a shift-invariant space of $L^2(\R^d)$. Then, there exists $\varphi\in V$ such that $\supp \|\mathcal T \varphi(\cdot)\|_{\ell^2}=\sigma(V).$
\end{lemma}

\begin{lemma}\label{range-properties}
Let $V, U$ be shift-invariant spaces of $L^2(\R^d)$ with associated range functions $J_V, J_U$ respectively. Then we have:
\begin{enumerate}[\rm (i)]
\item\label{item 1 range-properties} The orthogonal complement of $V$, $V^{\perp}$, is shift invariant and $J_{V^{\perp}}(\w)=\left(J_V(\w)\right)^{\perp}$ for a.e. $\w\in[0,1)^d$.
\item\label{item 2 range-properties} If   $J_{V}(\w)=J_{U}(\w)$ for a.e. $\w\in [0,1)^d$, then $U = V$.
\item\label{item 3 range-properties} The space  $V\cap U$ is  shift invariant and its range function satisfies $J_{V\cap\, U}(\omega) = J_V(\omega) \cap J_U(\omega),$
 	for a.e. $\omega\in [0,1)^d$.
\end{enumerate}
\end{lemma}

The last item of the above lemma was proved in \cite{AC}. The following Theorem is due by Bownik in \cite{B}.  

\begin{theorem}\label{thm:principal-decomposition}
Let $V$ be a shift-invariant space of $L^2(\R^d)$. Then $V$ can be decomposed as an orthogonal sum 
\begin{equation}\label{principal-decomposition}
V=\underset{i\in\N}{\bigpoplus}\, S(\varphi_i),
\end{equation} where $\varphi_i$ is a Parseval frame generator of $S(\varphi_i)$, and $\sigma(S(\varphi_{i+1}))\subset \sigma(S(\varphi_i))$ for all $i\in\N$. Moreover, $\dim J_{S(\varphi_i)}(\w) = \|\cT \varphi_i(\w)\|$ for every $i\in\N$ and $\dim J_V(\w) = \sum_{i\in\N} \|\cT\varphi_i(\w)\|$ for a.e. $\w\in [0,1)^d.$
\end{theorem}

Observe that $\varphi$ is a Parseval frame generator of $S(\varphi)$ if and only if $\|\cT \varphi(\w)\|=1_{\sigma(S(\varphi))}(\w)$ for a.e. $\w\in[0,1)^d$ (see Theorem 2.3 in \cite{B}).

The following lemma is a consequence of Theorem \ref{thm:principal-decomposition}, which will be a key to working through measurability problems in Section \ref{section-measurability}. We require here, as well as in many other places along this article, that $\dim J(\w)<\infty$ for a.e. $\w\in [0,1)^d$. We remark that this condition does not mean that $V$ must be finitely generated.  

\begin{lemma}\label{An}
	Let $V$ be a shift-invariant space with range function $J$ such that $\dim(J(\w))<\infty$ a.e. $\w\in [0,1)^d$. Then, there exist functions $\{\varphi_i\}_{i\in\N}$ of $L^2(\R^d)$ and a family of disjoint measurable sets $\{A_n\}_{n\in\N_0}$ such that $[0,1)^d=\bigcup_{n\in\N_0} A_n$ and the following statements hold:
	\begin{enumerate}[\rm(i)]
		\item $\{T_k\varphi_i\,:\,i\in\N,\, k\in\Z^d\}$ is a Parseval frame of $V$,
		\item $\cT\varphi_i(\w)=0$ for $i>n$ and  a.e. $\w\in A_n$,
		\item $\{\cT\varphi_1(\w),\dots,\cT\varphi_n(\w)\}$ is an orthonormal basis of $J(\w)$ for a.e. $\w\in A_n$,
		\item $\dim J(\w) = n$ for a.e. $\w\in A_n$.
	\end{enumerate}
\end{lemma}

\begin{proof}
	Let $\{\varphi_i\}_{i\in\N}\subset L^2(\R^d)$ be the functions from the decomposition of $V$ as in Theorem \ref{thm:principal-decomposition}. As $\varphi_i$ is a Parseval frame generator of $S(\varphi_i)$ for every $i\in\N$, it is easy to see that $\{T_k\varphi_i\,:\,i\in\N,\, k\in\Z^d\}$ is a Parseval frame of $V$. In addition, this  implies that $\{\cT \varphi_i(\w)\}_{i\in\N}$ is a Parseval frame of $J(\w)$ for a.e. $\w\in[0,1)^d$ (see Theorem 2.3 in \cite{B}).
	
	Define $A_0 = [0,1)^d\setminus \sigma(V)$. Now, for $n>0$ define the sets $A_n=\sigma(S(\varphi_n))\setminus \sigma(S(\varphi_{n+1}))$. Since $\sigma(S(\varphi_{i+1}))\subset \sigma(S(\varphi_i))$ for all $i\in\N$, the sets $\{A_n\}_{n\in\N_0}$ are pairwise disjoint. Moreover, given that $\dim J(\w)<\infty$ for a.e. $\w\in[0,1)^d$, the sum $\sum_{i\in\N} \|\cT\varphi_i(\w)\|$ must be finite for a.e. $\w\in [0,1)^d$ and so $\bigcap_{i\in\N}\sigma(S(\varphi_i))=\emptyset$. From this, we can finally conclude that $[0,1)^d=\bigcup_{n\in\N_0} A_n$.
	
	Fix $n>0$. By definition of the sets $A_n$, we have that (ii) is satisfied. Thus, using that (i) holds, the sum in \eqref{principal-decomposition} is orthogonal and $\cT$ is an isometry, we get that the system $\{\cT\varphi_1(\w),\dots,\cT\varphi_n(\w)\}$ must be an orthonormal basis of $J(\w)$ and so $\dim J(\w)=n$ for a.e. $\w\in A_n$. For $n=0$, $J(\w)=\{0\}$ for a.e. $\w\in A_0$.
\end{proof}

\begin{remark}
	Let $n\in\N$. Given a measurable set $B\subseteq [0,1)^d$ such that $\dim J(\w)=n$ for a.e. $\w\in[0,1)^d$, then $B\subseteq A_n$. In other words, $A_n=\{\w\in[0,1)^d\,:\, \dim J(\w)=n\}$.
\end{remark}

\section{Shift-preserving operators}\label{section-SP}

We describe here some properties of shift-preserving operators that we will use in the next section.

\begin{definition}
Let $V\subset L^2(\mathbb R^d)$ be a shift-invariant space and $L:V\rightarrow L^2(\mathbb R^d)$ be a bounded operator. We say that $L$ is {\it shift preserving} if $LT_k = T_k L$ for all $k\in\mathbb Z^d$.
\end{definition}

\begin{definition}\label{range operator}
Let $V$ be a shift-invariant space with range function $J$. A {\it range operator} on $J$ is a mapping
$$R: [0,1)^d\rightarrow \left\{ \text{bounded operators defined on closed subspaces of } \ell^2(\mathbb Z^d) \right\},$$
such that the domain of $R(\omega)$ is $J(\omega)$ for a.e. $\omega\in[0,1)^d$. 
\end{definition}

We say that $R$ is {\it measurable} if  $\omega \mapsto \langle R(\omega) P_{J(\omega)} u,v\rangle$ is a measurable scalar function for every $u,v\in\ell^2(\mathbb Z^d)$.

In \cite{B}, Bownik proved that given a bounded, shift-preserving operator $L:V\rightarrow L^2(\mathbb R^d)$, there exists a measurable range operator $R$ on $J$ such that 
\begin{equation}\label{prop of R}
(\mathcal T\circ L) f(\omega) = R(\omega) \left(\mathcal T f(\omega)\right),
\end{equation}
for a.e. $\omega\in\mathbb [0,1)^d$ and $f\in V$.

Conversely, given a measurable range operator $R$ on $J$ with $$\underset{\omega\in [0,1)^d}{\esssup} \|R(\omega)\|<\infty,$$ there exists a bounded shift preserving operator $L: V\to L^2(\R^d)$ such that  \eqref{prop of R} holds. The correspondence between $L$ and $R$ is one-to-one under the convention that the range operators are identified if they are equal a.e. $\w\in[0,1)^d$. Furthermore, he also proved the following relation between the operator norm of $L$ and its associated range operator, 
\begin{equation}\label{norm of R and L}
\|L\| = \underset{\omega\in [0,1)^d}{\esssup} \|R(\omega)\|.
\end{equation}

Let us now focus on the particular case where $L:V\rightarrow V$. In this setting, $R(\omega):J(\omega)\rightarrow J(\omega)$ for a.e. $\omega\in[0,1)^d$. The following result, that was proved in \cite{B}, shows that some properties of the operator $L$ are intrinsically related to the properties of its fibers. 

\begin{theorem}\label{adjoint L}
Let $V$ be a shift-invariant space and $L:V\rightarrow V$ a shift-preserving operator with associated range operator $R$. Then, the adjoint operator $L^*:V\rightarrow V$ is also a shift-preserving operator and its corresponding range operator $R^*$ satisfies that $R^*(\omega) = (R(\omega))^*$ for a.e. $\omega\in [0,1)^d$.  As a consequence,  $L$ is self-adjoint if and only if  $R(\omega)$ is self-adjoint for a.e. $\omega \in [0,1)^d$, and  $L$ is a normal operator if and only if  $R(\omega)$ is a normal operator for a.e. $\omega \in [0,1)^d$.  
\end{theorem}

When a shift-preserving operator acts on a  shift-invariant space satisfying that
	$\dim J(\w)<\infty$ for a.e. $\w\in [0,1)^d$, the fibers of its corresponding range operator can be identified with matrices. This is done explicitly in the next proposition.
	
	\begin{proposition}\label{matrix R}
		Let $V$ be a shift-invariant space with range function $J$ such that $\dim(J(\w))<\infty$ a.e. $\w\in [0,1)^d$ and $L:V\rightarrow V$ a shift-preserving operator with associated range operator $R$. Then, $R(\w)$ has a matrix representation for a.e. $\w\in [0,1)^d$ with measurable entries. More precisely, if $\dim(J(\w))=n$ a.e. $\w\in B$ where $B\subseteq [0,1)^d$ is measurable, then, there exist $n^2$ measurable bounded functions defined on $B$, $\{m_{i,j}\}_{i,j=1}^n$ such that the matrix representation of $R(\w)$ is 
		\begin{equation}\label{eq:matrix}
		\begin{bmatrix}
		m_{1,1}(\omega)  & m_{1,2}(\omega) & \cdots & m_{1,n}(\omega) \\
		m_{2,1}(\omega)  & m_{2,2}(\omega) & \cdots & m_{2,n}(\omega) \\
		\vdots     &           & \ddots &               \\
		m_{n,1}(\omega)  & m_{n,2}(\omega) & \cdots & m_{n,n}(\omega) \\
		\end{bmatrix},
		\end{equation}
		for a.e. $\omega\in B$.
	\end{proposition}
	\begin{proof}
		Let $\{\varphi_i\}_{i\in \N}$ and $\{A_n\}_{n\in\N_0}$ as in Lemma \ref{An}.
		For a fixed  $n\in\N$, we see that for a.e. $\w\in A_n$, $R(\w)$ has a matrix representation $[R](\w)$
		relative to the orthonormal basis  $\{\mathcal T\varphi_1(\omega),\dots,\mathcal T\varphi_n(\omega)\}$ of $J(\w)$ with entries 
		$$m_{i,j}(\w):=([R\,](\omega))_{i,j} =\left\langle R(\w)\mathcal T\varphi_j(\w),\mathcal T\varphi_i(\w) \right\rangle,$$
		which are clearly measurable. Notice that since $L$ is a bounded operator and (\ref{norm of R and L}) holds, then $|m_{i,j}(\w)|\leq \|L\|$ for $i,j=1,\dots,n$ and for a.e. $\w\in A_n$.
		
		Given that every $B\subseteq [0,1)^d$ where $\dim J(\w) = n$ for a.e. $\w\in B$ is included in $A_n$, the statement of the proposition holds.
		\end{proof}
	
	We state now two measurability results related to the range operator, whose proofs we postpone until Subsection \ref{section-measurability-1}.
	
	\begin{proposition}\label{measurable-prop-range}
		Let $R(\omega):J(\omega)\rightarrow J(\omega)$ be a measurable range operator on a range function $J$ satisfying $\dim J(\w)<\infty$ a.e. $\w\in [0,1)^d$. Then the following statements hold:
		\begin{enumerate}[\rm(i)]
			\item The mapping $\omega \mapsto\ker(R(\omega))$, $\omega\in [0,1)^d$ is a measurable range function.
			\item\label{measurable-prop-range-2} If $R(\w)$ is invertible for a.e. $\w\in [0,1)^d$, then $\omega \mapsto (R(\omega))^{-1}$, $\omega\in [0,1)^d$ is a measurable range operator.
		\end{enumerate}
	\end{proposition}

		In the same spirit of Theorem \ref{adjoint L}, we will show that when $\dim(J(\w))<\infty$ for a.e. $\w\in [0,1)^d$, the invertibility of a shift-preserving operator can be
		deduced from the invertibility of its fibers.

\begin{theorem}\label{L^-1 sp}
	Let $V$ be a shift-invariant space with range function $J$ such that $\dim(J(\w))<\infty$ for a.e. $\w\in [0,1)^d$ and $L:V\rightarrow V$ a shift-preserving operator with associated range operator $R$. Then, the following statements hold:
	\begin{enumerate}[\rm(i)]
		\item If $L$ is invertible, then $L^{-1}$ is also shift preserving.
		\item $L$ is invertible if and only if $R(\w)$ is invertible for a.e. $\w\in[0,1)^d$ and there exist a constant $C>0$ such that $R(\w)$ is uniformly bounded from below by $C$ (i.e. for a.e. $\w\in[0,1)^d$ we have that $\|R(\w)a\|\geq C\|a\|$ for all $a\in J(\w)$). In that case, if we denote by $R^{-1}$ the range operator associated to $L^{-1}$, we have that $(R(\w))^{-1}=R^{-1}(\w)$ for a.e. $\w\in[0,1)^d$.
	\end{enumerate}  
\end{theorem}

\begin{proof}
	Assume that $L$ is invertible. Then, the inverse of $L$ is a bounded operator. Now, for every $f\in V$ and $k\in \Z^d$ we see that
	$$L^{-1}T_k f = L^{-1}T_k LL^{-1} f = L^{-1} L T_k L^{-1}f = T_k L^{-1}f,$$
	then $L^{-1}$ is a shift-preserving operator. This proves (i).
	
	For (ii), first assume that $L$ is invertible and denote by $R^{-1}$ the range operator of $L^{-1}$. It is not difficult to prove that $\w\mapsto R^{-1}(\w)R(\w)$ and $\w\mapsto R(\w)R^{-1}(\w)$ are measurable and uniformly bounded range operators on $J$. Furthermore, their associated shift-preserving operators are $L^{-1}L$ and $LL^{-1}$ respectively.
	Hence, we have that $R^{-1}(\w)R(\w) = R(\w)R^{-1}(\w) = \mathcal {I}_\omega$, where $\mathcal {I}_\omega$ denotes the identity range operator on $J(\w)$ for a.e. $\w\in [0,1)^d$. 
	Thus, $R(\w)$ is invertible and $(R(\w))^{-1} = R^{-1}(\w)$ for a.e. $\w\in [0,1)^d$. In addition, $R(\w)$ is uniformly bounded from below, since $(R(\w))^{-1}$ is bounded uniformly from above by $\|L^{-1}\|$.
		
	For the converse, assume that $R(\w)$ is invertible for a.e. $\w\in [0,1)^d$ and is uniformly bounded from below by $C>0$. Then $\w\mapsto (R(\w))^{-1}$ is a measurable range operator on $J$ (by (ii) in Proposition \ref{measurable-prop-range}) and is uniformly bounded by $C$. Hence, it has a corresponding shift-preserving operator $\tilde{L}$ such that for every $f\in V$ and a.e. $\w\in [0,1)^d$, $(R(\w))^{-1}(\mathcal T f (\w)) = \mathcal T\tilde{L} f (\w)$. Given that 
	\begin{align*}
		\mathcal T f(\w) = (R(\w))^{-1}R(\w)(\mathcal T f(\w)) = \mathcal T \tilde{L}L f(\w),\\
		\mathcal T f(\w) = R(\w)(R(\w))^{-1}(\mathcal T f(\w)) = \mathcal T L\tilde{L} f(\w),
	\end{align*}
	for every $f\in V$ and a.e. $\w\in [0,1)^d$, we deduce that $\tilde{L}L = L\tilde{L}=\mathcal I_V$, and so $L$ is invertible and $\tilde{L}=L^{-1}$.
\end{proof}

\begin{remark}
	The condition of $R(\w)$ being uniformly bounded from below is natural as we know that, in fact, $L$ is bounded from below by a constant $C>0$ (i.e. $\|Lf\|\geq C\|f\|$ for all $f\in V$), if and only if $R(\w)$ is uniformly bounded from below by $C$ (see \cite{B}, Theorem 4.6).
	For example, for $d=1$, let $R$ be a range operator on a range function $J$ defined by $R(\w)a=\w a$ for $a\in J(\w)$ and $\w\in [0,1)$. This range operator is measurable and is  uniformly bounded from above, thus it has a corresponding shift-preserving operator $L$. Notice that $R(\w)$ is invertible for a.e. $\w\in [0,1)$. However, since it is not uniformly bounded from below, then $L$ is not bounded from below and is not invertible.
\end{remark}

We will now show some examples of shift-preserving operators and their corresponding range operator. The first example was shown in \cite{B}.

\begin{example}\label{ex:frame}
Let $I$ be an index set. Suppose $\Phi = \{\varphi_i\,:\,i\in I\}$ is a set of functions such that $E(\Phi)$ is a Bessel family. Then, the frame operator of $E(\Phi)$ is self-adjoint and shift preserving with corresponding range operator $R(\omega)$ given by the frame operator of $\{\mathcal T \varphi_i (\omega)\,:\,i\in I\}$ for a.e. $\omega\in [0,1)^d$.
\end{example}

\begin{example}\label{ex:projection}
Let $V$ be a shift-invariant space of $L^2(\R^d)$ with range function $J$. Denote by $P_V:L^2(\R^d)\to L^2(\R^d)$ the orthogonal projection of $L^2(\R^d)$ onto $V$. Then, $P_V$ is shift preserving and, by Lemma \ref{Helson projections}, its range operator is given by $R(\w)=P_{J(\w)}$ for a.e. $\w\in  [0,1)^d$. 
\end{example}
For the following example we need to introduce a new definition that will become crucial later on. 
 
\begin{definition}
	We say that a sequence $a=\{a_j\} \in \ell^{2}(\mathbb{Z}^d)$ is of {\it bounded spectrum} if $\hat{a} \in L^{\infty}([0,1)^{d})$, where $\hat{a}(\omega)= \sum_{j\in \mathbb{Z}^{d}} a_{j}e_j(\omega)$.
\end{definition}

\begin{example}\label{ex:convolution}
	Let $\Phi=\{ \varphi_1,...,\varphi_n\} \subset L^{2}(\mathbb{R}^d)$ be a set of functions such that $V=S(\Phi)$ and $E(\Phi)$ is a Riesz basis of $V$. For every $f\in V$, there exist unique sequences $b_{1},...b_{n}$ in $\ell^{2}(\mathbb{Z}^d)$ such that $f=\sum_{i=1}^{n} \sum_{k\in \mathbb{Z}^{d}} b_{i}(k) T_{k}\varphi_{i}$.
	Now, let $a_{1},...,a_{n}$ be sequences of bounded spectrum and let us define an operator $L:V\rightarrow V$ by
	\begin{equation}\label{eq:op-convolution}
	Lf := \sum_{i=1}^{n} \sum_{k\in \mathbb{Z}^{d}} (b_{i}\ast a_{i})(k) T_{k}\varphi_{i}.
	\end{equation}
	Then, $L$ is a bounded shift-preserving operator, and its range operator can be written as an $n\times n$ diagonal matrix with diagonal $\hat{a}_1(\omega),...,\hat{a}_n(\omega)$ in the basis $\{\mathcal{T}\varphi_{1}(\omega),...,\mathcal{T}\varphi_{n}(\omega)\}$ for a.e. $\omega\in [0,1)^{d}$.
\end{example}

   \begin{remark}
Given a shift-invariant space $V$ with range function $J$ and a shift-preserving operator $L:V\rightarrow V$, define the operator $\widetilde{L}:=\mathcal T L \mathcal T^{-1}: M_J\rightarrow M_J$, where $\mathcal T$ is the map defined in the Proposition \ref{isometria}. Then,  $M_J$ can be written as a {\it direct integral Hilbert space} and the operator $\widetilde{L}$ as a  {\it direct integral operator}, (see \cite{Dix} for definitions).

That is 

$$M_J = \int_{[0,1)^d}^{\oplus} J(\w) \,d\w, \quad  \quad{\text and }\quad\quad  \widetilde{L} = \int_{[0,1)^d}^{\oplus} R(\w) \,d\w, $$
where $R$ is the range operator related  to  $L.$ 

Thus, most of the results of this paper can be extended to the general case of direct integral operators acting on direct integral Hilbert spaces and hyper-invariant subspaces,
not necessarily associated with shift-invariant spaces. 
   \end{remark}

\section{$s$-eigenvalues and $s$-eigenspaces}\label{section-s-eigenvalues}

In this section, we introduce the concepts of {\it $s$-eigenvalue} and {\it $s$-eigenspace} for a shift-preserving  operator $L$ and we study the relation between these and the eigenvalues and eigenspaces of its range operator.
This requires to deal with delicate issues concerning the measurability of the eigenvalues of $R(\w)$ and their associated eigenspaces that we will study in Section \ref{section-measurability}.

We begin with  a proposition that will serve as a starting point to define $s$-eigenvalues of $L$.

For a function $\psi$  defined on $[0,1)^d$ with values in $\CC$, we define the multiplication operator $M_\psi :L^2([0,1)^d,\ell^2(\Z^d)) \rightarrow L^2([0,1)^d,\ell^2(\Z^d))$ as $M_\psi f =  \psi f .$ It is well known that $M_\psi $ is well defined and bounded if and only if $\psi \in L^{\infty}([0,1)^d)$.
Thus, we have the following proposition.

\begin{proposition}\label{bd Lambda_a}
Given $a=\{a_k\}_{k\in\Z^d}\in\ell^2(\mathbb Z^d)$,  the operator 
\begin{equation*}
\Lambda_a  = \mathcal T^{-1}M_{\hat a}\mathcal T
\end{equation*}
 is well-defined and bounded  from $L^2(\R^d)$ into $L^2(\R^d)$  if and only if the sequence $a$ is of bounded spectrum. 
\end{proposition}

\begin{proof}
The result is a consequence of the fact that $M_{\hat a}$ and $\Lambda_a$ are unitarily equivalent via the operator $\mathcal T.$  
 \end{proof}
Define $\mathcal{B}$  as the set of functions  $g\in L^2(\R^d)$ such that $\{T_kg\}_{k\in\Z^d}$ is a Bessel sequence. We note that the operators 
$\Lambda_a $ and 
$\sum_{k\in\mathbb Z^d} a_k T_k  $ coincide  for every element of $\mathcal{B}$.
 Since $\mathcal{B}$ is dense in $L^2(\R^d),$ we have the following  alternative definition for $\Lambda_a$: define $\Lambda_a$ as $\sum_{k\in\mathbb Z^d} a_k T_k $  for the elements in $\mathcal{B}$ and then extend it continuously from $\mathcal{B}$ to $L^2(\R^d)$. 
Hence, we will write
 formally  $\Lambda_a = \sum_{k\in\mathbb Z^d} a_k T_k.$
 
 \begin{remark}
 If $a\in\ell^2(\Z^d)$ is of bounded spectrum, then $\Lambda_a$ satisfies the intertwining property $\mathcal T(\Lambda_af)(\w)=\hat{a}(\w)\mathcal T f(\w)$ for every $f\in L^2(\R^d)$ and for a.e. $\w\in[0,1)^d$.

 Observe that, if $V$ is  shift invariant, $\Lambda_a: V\to V$  is a  shift-preserving  operator and its corresponding range operator is $R_a(\omega)=\hat{a}(\omega)\mathcal {I}_\omega$ for a.e. $\omega\in [0,1)^d$, where $\mathcal {I}_\omega$ denotes the identity operator on $J(\w)$.

Moreover, when $E(\Phi)$ is a frame of $V$ where $\Phi=\{\varphi_1,\dots, \varphi_n\}\subset L^2(\R^d)$ we have that every $f\in V$ can be written as $f=\sum_{i=1}^{n} \sum_{k\in \mathbb{Z}^{d}} b_{i}(k) T_{k}\varphi_{i}$ with $b_i\in\ell^2(\Z^d)$ for all $i=1,\dots,n$.  Then, 
$$\widehat{\Lambda_a f}(\w)=\hat{a}(\w)\sum_{i=1}^n\hat{b}_i(\w)\widehat{\varphi}_i(\w)=\sum_{i=1}^n\widehat{a* b_i}(\w)\widehat{\varphi}_i(\w)\quad\textrm{for a.e. }\w\in\R^d$$
and thus, $\Lambda_a f= \sum_{i=1}^n \sum_{k\in\Z^d}(a* b_i)(k)T_k\varphi_i$.
 \end{remark}

  \begin{definition}\label{def: s-eigenvalues}
  	Let $V$ be a shift-invariant space and $L:V\rightarrow V$ a bounded shift-preserving operator. Given $a\in \ell^{2}(\mathbb{Z}^d)$ a sequence of bounded spectrum, we say that $\Lambda_a$ is an {\it $s$-eigenvalue} of $L$ if 
	\begin{equation*} 
  	V_a := \ker\left(L - \Lambda_a\right)\neq\{0\}.
  	\end{equation*}
  	We call $V_a$ the {\it $s$-eigenspace} associated to $\Lambda_a$.
 \end{definition} 
 
   \begin{remark}
   	Observe that if $\lambda\in\mathbb C$ is an eigenvalue of $L$ (i.e. $\ker(L-\lambda\mathcal I)\neq \{0\}$), then it is an $s$-eigenvalue $\Lambda_a$ of $L$, taking the sequence $a=\lambda e_0$, where $e_0(k)=\delta_{0,k}$, $k\in\Z^d$. Thus, $s$-eigenvalues generalize the eigenvalues of $L$.
   \end{remark}
   
 Notice that $V_a$ is always a shift-invariant subspace of $V$ and it is $L$-invariant, that is,  $LV_a\subseteq V_a.$
 Also, observe that since for every $f\in V_a$, $Lf=\Lambda_af,$  then
 \begin{equation}\label{intertwining-prop} 
 R(\omega)(\mathcal Tf(\omega)) = \mathcal T(Lf)(\omega) = \cT(\Lambda_a f)(\w) = \hat{a}(\w)\cT f(\w),
 \end{equation}
 for a.e. $\omega\in [0,1)^d.$ Hence, this new  definition of $s$-eigenvalue is related with the eigenvalues of the range operator of $L$, as we state in the next proposition. 
 
 \begin{proposition}\label{prop-eigen}
 Let $V$ be a shift-invariant space with range function $J$ such that $\dim J(\w)<\infty$ for a.e. $\w\in [0,1)^d$, $L:V\rightarrow V$ a bounded shift-preserving operator with associated range operator $R$ and  $a\in \ell^{2}(\mathbb{Z}^d)$ a sequence of bounded spectrum. Then, the following statements hold:
 \begin{enumerate}[\rm(i)]
\item \label{eigenvalues}If $\Lambda_a$ is an $s$-eigenvalue of $L$, then $\lambda_a(\omega) := \hat{a}(\omega)$ is an eigenvalue of $R(\omega)$ for a.e. $\omega\in \sigma(V_a)$.
\item  \label{eigenspaces} The mapping $\omega \mapsto \ker\left(R(\omega) - \lambda_a(\omega)\mathcal {I}_\omega\right)$, $\omega \in [0,1)^d$ is the measurable range function of $V_a$, which we will denote $J_{V_a}$.
 \end{enumerate}
 \end{proposition}
 
 \begin{proof}
 (i). By Lemma \ref{spec principal}, there exists $\varphi_a\in V_a$ such that $\supp \|\mathcal T \varphi_a(\cdot)\|_{\ell^2}=\sigma(V_a).$
By \eqref{intertwining-prop},  $\mathcal T \varphi_a (\omega) \in \ker\left(R(\omega) - \lambda_a(\omega)\mathcal {I}_\omega\right)$ for a.e. $\omega\in[0,1)^d$. Now, given that $\mathcal T \varphi_a (\omega) \neq 0$ almost everywhere in $\sigma(V_a)$, we have  that for a.e. $\omega\in \sigma(V_a)$,
 	$$\ker\left(R(\omega) - \lambda_a(\omega)\mathcal {I}_\omega\right)\neq\{0\},$$
 	i.e. $\lambda_a(\omega)$ is an eigenvalue of $R(\omega)$ for a.e. $\omega\in \sigma(V_a)$.
	
(ii). It is clear that the mapping $\w\mapsto R(\w) - \lambda_a(\w)\mathcal {I}_\omega$ is a measurable range operator. Hence, by (i) in Proposition \ref{measurable-prop-range}, the mapping $\w\mapsto \ker\left(R(\omega) - \lambda_a(\omega)\mathcal {I}_\omega\right)$ is a measurable range function.	
	
 	It remains to see that it is, in fact, the range function associated to $V_a$. But this is a straightforward consequence of \eqref{range function} and  \eqref{intertwining-prop}.
 \end{proof}
  
  Note that we have proved that the range function associated to $V_a$ takes as values the eigenspaces of $R(\omega)$ associated to the eigenvalues $\lambda_a(\omega)$ for a.e. $\omega\in\sigma(V_a)$. 
       
 On the other hand, measurable eigenvalues of $R(\omega)$ induce $s$-eigenvalues of $L$, as we state in the following lemma.
 
 \begin{lemma}\label{eigenvalue of R}
 Let $V$ be a shift-invariant space and  $L:V\rightarrow V$ a bounded shift-preserving operator with range operator $R$.  Let $B\subseteq [0,1)^d$ be a measurable set of positive measure. Suppose  $\lambda:B\rightarrow \mathbb C$ is a measurable function such that $\lambda(\omega)$ is an eigenvalue of $R(\omega)$ for a.e. $\omega\in B$, then there exists $a\in \ell^2(\mathbb Z^d)$ of bounded spectrum such that $\lambda(\omega)=\hat{a}(\omega)$ a.e. $\w\in B$ and $\Lambda_a$ is an $s$-eigenvalue of $L$.
 \end{lemma}
 
 \begin{proof}
Observe that as $R$ is the range operator of $L$, by (\ref{norm of R and L}) we get that $|\lambda(\omega)|\leq \|L\|$ for a.e. $\omega\in B$. Consider any bounded extension of $\lambda$ to $[0,1)^d$, then, we have that $\lambda\in L^\infty([0,1)^d)$ and thus there exist $a\in\ell^2(\mathbb Z^d)$ of bounded spectrum such that $\hat{a}(\w) = \lambda(\w)$. On the other hand, since $\lambda(\omega)$ is an eigenvalue of $R(\omega)$ for a.e. $\w\in B$, then $$\ker\left(R(\omega) - \lambda(\omega)\mathcal {I}_\omega\right)\neq\{0\}$$
for a.e. $\w\in B$, from which we deduce that $V_a\neq \{0\}$, and therefore $\Lambda_{a}$ is an $s$-eigenvalue of $L$.
 \end{proof}
 
A natural question that arises here is whether these kind of measurable functions exist. 
When $\dim J(\w)$ is constant, we can consider  $R(\w)$ as a matrix of measurable entries (see  Proposition \ref{matrix R}). We will prove the existence of measurable eigenvalues for that case in Subsection \ref{measurable-eigenvalues}. Then, for the more general case on which $\dim J(\w)<\infty$ for a.e. $\w\in [0,1)^d$ we will construct measurable functions and measurable sets on which such functions are eigenvalues of $R(\w)$ (see Theorem \ref{thm:autovalores-full}). 

 In the next proposition we give conditions on two $s$-eigenvalues in order that its associated $s$-eigenspaces are in direct sum. 
 
 \begin{proposition}\label{different eigenvalues}
  Let $V$ be a shift-invariant space and  $L:V\rightarrow V$ a bounded shift-preserving operator.
 	Let $a,b \in \ell^2(\Z^d)$ be two distinct sequences of bounded spectrum such that $\Lambda_a$ and $\Lambda_b$ are $s$-eigenvalues of $L$. The following statements are equivalent:
 	\begin{enumerate}[\rm (i)]
	 	\item $V_a\cap V_b =\{0\}$,
	 	\item $\hat{a}(\omega)\neq\hat{b}(\omega)$ almost everywhere in $\sigma(V_a)\cap \sigma(V_b)$.
 	\end{enumerate}
 \end{proposition}
 \begin{proof}
 We assume that $\sigma(V_a)\cap \sigma(V_b)\neq \emptyset$ since, otherwise, the equivalence holds.
 
 $(\rm i)\Rightarrow (\rm ii)$. 
  By (iii) in Lemma \ref{range-properties}, we have that $J_{V_a}(\omega)\cap J_{V_b}(\omega) = J_{V_a\cap V_b}(\omega) = \{0\}$, i.e. 
 	$$\ker\left(R(\omega)-\lambda_a(\omega)\mathcal {I}_\omega\right)\cap\ker\left(R(\omega)-\lambda_b(\omega)\mathcal {I}_\omega\right)=\{0\},$$
 	 for a.e. $\omega \in [0,1)^d$.
 	 
 	 Suppose there exists a measurable set $A\subset \sigma(V_a)\cap\sigma(V_b)$ such that $|A|>0$ and $\hat{a}(\omega) = \hat{b}(\omega)$ when $\omega\in A$. Then, if $\omega\in A$, we have that
 	 $$\ker\left(R(\omega)-\lambda_a(\omega)\mathcal {I}_\omega\right)=\ker\left(R(\omega)-\lambda_b(\omega)\mathcal {I}_\omega\right)=\{0\},$$
 	 which is a contradiction.
 	 
 	  $(\rm ii)\Rightarrow (\rm i)$. Let $f\in V_a\cap V_b$, then $Lf=\Lambda_af=\Lambda_bf$ and thus, 
	   	 $$(\hat{a}-\hat{b})(\omega)\cT f(\omega)= 0,$$
 for a.e. $\w\in[0,1)^d$.
 	By (ii), we get that $\cT f(\w)=0$ for a.e. $\omega \in \sigma(V_a)\cap \sigma(V_b)$, given that $\sigma(V_a\cap V_b)\subseteq \sigma(V_a)\cap \sigma(V_b)$, we have that $f=0$.
 \end{proof}
 
 \section{Measurability}\label{section-measurability}
 
This section is dedicated to discussing  measurability questions that arose in the previous sections. It is divided into two subsections. In the first one, we prove Proposition \ref{measurable-prop-range} which shows the measurability of the kernel and the inverse of a measurable range operator. The second subsection is devoted to studying the existence of measurable eigenvalues for a range operator.

\subsection{Proof of Proposition\ \ref{measurable-prop-range}}\label{section-measurability-1}  

For the proof of (i) in Proposition \ref{measurable-prop-range}, we need the subsequent two lemmas which were proved in \cite{ACHM}. For the sake of completeness we will include their proofs.

\begin{lemma}\label{set of rank k}
	Let $M=M(\omega)$ be an $n\times n$ matrix of measurable entries  defined on a measurable set $B\subset \R^d$. If for $k=0,\dots, n$ $$B_k :=\{\omega\in B \,:\, {\text{\rm rank}}(M(\omega)) = k\},$$  then $B_k$ is measurable and $B=\bigcup_{k=0}^n B_k$.
\end{lemma}

\begin{proof}
	For $k=0,\dots,n$ consider $C_k(\omega) := \{A(\omega) \in \mathbb C^{k\times k}: A(\omega)\,\text{ submatrix of } M(\omega)\}$, and let $f_k:B \rightarrow \R$ be defined as
	$$ f_k(\omega) = \sum_{A(\omega) \in C_k(\omega)} |\det(A(\omega))|.$$
	Then, $f_k$ is measurable for each $k$, and $f_k(\omega) > 0$ if and only if $\text{rank }(M(\omega)) \geq k$.  Since 
	$B_k = \{\omega\in B\,:\, f_k(\omega) > 0\} \setminus \{\omega\in B\,:\, f_{k+1}(\omega) > 0\}$, it follows that $B_k$ is  measurable.
\end{proof}

\begin{lemma}\label{measurable ker}
	Let $B\subset\mathbb R^d$ be a measurable set and $M = M(\omega)$ be an $n\times n$ matrix of measurable entries defined on $B$. Suppose that $\rank \left(M(\w)\right)$ = k for a.e. $\w\in B$.  Then there exist measurable functions $v_1,\dots , v_{n-k} : B \rightarrow \mathbb C^n$ such that the vectors $\{v_1(\omega), \dots , v_{n-k}(\omega)\}$ form an orthonormal basis for $\ker\left(M(\omega)\right)$, for a.e. $\omega\in B$.
	In particular, if $P_{\ker\left(M(\omega)\right)}$ denotes the orthogonal projection of $\mathbb C^n$ onto $\ker\left(M(\omega)\right)$, we have that 
	$$B\ni\w\mapsto\langle P_{\ker\left(M(\omega)\right)}x,y\rangle$$
	is measurable for every $x,y\in\mathbb C^n$.
\end{lemma}

\begin{proof}
	Define $$S_k = \{C\times D : \, C\subset \{1,\dots,n\}, \, D \subset \{1,\dots,n\}, \#C=\#D=k\}. $$ 
	For $h \in S_k$, let $M_h$ be the submatrix of $M$ whose entries have subindexes in $h$. We have that $B = \bigcup_{h \in S_k} B_h$, where $B_h = \{\omega \in B: \det(M_h(\omega)) \not= 0\}.$ The continuity of the determinant function implies that $B_h$ are measurable for all $h \in S_k$. From the sets $B_h$ we can construct measurable disjoint  sets $E_h, \, h \in S_k$ such that $B = \bigcup_{h \in S_k} E_h$ and $\det(M_h(\omega)) \not=0 \;  \text{ for a.e. }\omega \in E_h$.
	In this way, for each $\omega$ we selected a unique invertible $k \times k$ submatrix $M_h(\omega)$. Note that for a fixed $h \in S_k$ the indexes of the submatrices $M_h(\omega)$ are the same for
	all $\omega$.
	
	For each fixed $h \in S_k$ we want to choose measurable functions 
	$v_1, \dots, v_{n-k}$ defined on $E_h$,  such that the vectors $\{v_1(\omega),\dots, v_{n-k}(\omega)\}$ form a basis of $\ker(M(\omega))$ for a.e. $\omega \in E_h$.
	
	Without loss of generality we may assume that $h = \{1,\dots,k\}\times\{1,\dots,k\}$ (the general case can be reduced to this one by permutation of rows and columns).
	For $\omega \in E_h$ let us define the vector functions:
	\begin{align*}
	v_1^h & = (x_{11}^h,\dots,x_{1k}^h,1,0,\dots,0)\\
	v_2^h & = (x_{21}^h,\dots,x_{2k}^h,0,1,0,\dots,0)\\
	\vdots \\
	v_{n-k}^h & = (x_{(n-k)1}^h,\dots,x_{(n-k)k}^h,0,\dots,0,1)
	\end{align*}
	with $( x_{i1}^h(\omega),\dots,x_{ik}^h(\omega)) = -M^{-1}_h(\omega) c_{i+k}(\omega)$ for $i=1,\dots,n-k$, where
	$c_{i+k}$ is the column vector containing the first $k$ entries of the ${(i+k)}$-column of $M(\omega)$.
	It is straightforward to see that the vectors $\{v_1^h(\omega),\dots, v_{n-k}^h(\omega)\}$ are measurable functions and form a basis of $\ker(M(\omega))$ for a.e. $\omega \in E_h$.

	Finally, if we set $v_i = \sum_{h\in S_k} v_i^h\chi_{E_h}, \, i = 1,\dots, m-k$,  we obtain a measurable basis for $\ker(M(\omega)),\text{ for a.e. }\, \omega \in B$. Since the Gram-Schmidt orthogonalization process  conserves measurability, we obtain the desired result.  
\end{proof}

\begin{proof}[Proof of Proposition \ref{measurable-prop-range}] 
	
	Let $\{\varphi_i\}_{i\in\N} \subset L^2(\R^d)$ and the sets $\{A_n\}_{n\in\N_0}$ be the ones given by Lemma \ref{An}.
	
	(i).  We will prove that $\w\mapsto\ker(R(\omega))$ is a measurable range function on $[0,1)^d$ by showing that it is measurable on each $A_n$ for every $n\in\N_0$.
	
	For $A_0$ there is nothing to prove. Fix $n\geq 1$. We need to prove 
	that the scalar function $A_n\ni\w\mapsto\langle P_{\ker(R(\omega))}u,v\rangle$ is measurable for every $u,v\in\ell^2(\Z^d)$. 
	By Lemma \ref{An}, we have that $\{\mathcal T\varphi_1(\omega),\dots, \mathcal T\varphi_n(\omega)\}$ is an orthonormal basis of $J(\w)$ for a.e. $\w\in A_n$. Then, it is enough to see that $\omega\mapsto \langle P_{\ker (R(\omega))} \mathcal T \varphi_i(\omega),\mathcal T \varphi_{j}(\omega)\rangle$ is measurable for every $i,j=1\dots,n$. 
	
	Let $\{e_{1},\dots,e_{n}\}$ be the canonical basis of $\mathbb C^n$ and, for $\w\in A_n$, define $\nu(\omega):J(\omega)\rightarrow \mathbb C^n$  as 
	\begin{equation}\label{eq:cambio-de-base}
	\nu(\omega)\left(\mathcal T \varphi_i(\omega)\right) = e_i,
	\end{equation}
	for $i=1,\dots,n$. This is the change-of-basis operator from $\{\mathcal T\varphi_1(\omega),\dots,\mathcal T\varphi_n(\omega)\}$ to $\{e_1,\dots, e_n\}$, which is unitary for a.e $\w\in A_n$. 
	Then $$P_{\ker(R(\omega))} = \nu(\omega)^{-1}P_{\ker\left([R\,](\omega)\right)}\nu(\omega),$$ where $[R\,](\w)$ is the matrix of $R(\w)$ relative to the basis $\{\mathcal T\varphi_1(\omega),\allowbreak\dots,\mathcal T\varphi_n(\omega)\}$. 
	Therefore, given $i$ and $j$, we have that
	\begin{align*}
	\left\langle P_{\ker (R(\omega))} \mathcal T \varphi_i(\omega),\mathcal T \varphi_{j}(\omega)\right\rangle &= \left\langle \nu(\omega)^{-1}P_{\ker\left([R\,](\omega)\right)}\nu(\omega) \mathcal T \varphi_i(\omega),\mathcal T \varphi_{j}(\omega)\right\rangle\\
	&=\left\langle P_{\ker\left([R\,](\omega)\right)}e_j,e_{j}\right\rangle,
	\end{align*}
	for a.e. $\w\in A_n$. 
	
	Consider now the partition of $A_n$ into the sets $\{B_1,\cdots,B_n\}$ given by Lemma \ref{set of rank k}. Then, on each $B_k$, $[R\,](\omega)$ is a matrix of rank $k$ and, by Lemma \ref{measurable ker}, $B_k\ni\w\mapsto\langle P_{\ker\left([R\,](\omega)\right)}e_i,e_{j}\rangle$ is a measurable scalar function for every $1\leq k \leq n$. Then, so is $A_n\ni\w\mapsto\langle P_{\ker\left([R\,](\omega)\right)}e_i,e_{j}\rangle$ and this is what  we wanted to prove.

 (ii). As before, in order to prove that $\w\mapsto (R(\w))^{-1}$ is a measurable range operator on $[0,1)^d$, it suffices to prove the measurability on each $A_n$ for every $n\in\N_0$.
 
 It is clearly true for $A_0$. Now, for a fixed $n\geq 1$ we need to prove that $A_n\ni\w\mapsto \langle (R(\w))^{-1}P_{J(\w)} u, v\rangle$ is measurable for every $u,v\in\ell^2(\Z^d)$. 
  Again, it is enough to see that $\omega\mapsto \left\langle (R(\w))^{-1}\mathcal T \varphi_i(\omega),\mathcal T \varphi_{j}(\omega)\right\rangle$ is measurable for every $i,j=1\dots,n$. 
  
  Observe that $\nu(\w)(R(\w))^{-1}\nu(\w)^{-1} = ([R\,](\w))^{-1}$, where $\nu(\omega)$ is the change-of-basis defined in \eqref{eq:cambio-de-base}. Hence, given $i$ and $j$, we see that
 \begin{align*}
 \left\langle R(\w)^{-1}\mathcal T \varphi_i(\omega),\mathcal T \varphi_{j}(\omega)\right\rangle&=\left\langle \nu(\w)^{-1}([R\,](\w))^{-1}\nu(\w)\mathcal T \varphi_i(\omega),\mathcal T \varphi_{j}(\omega)\right\rangle\\
 &=\left\langle ([R\,](\w))^{-1} e_i, e_{j}\right\rangle,
 \end{align*}
 for a.e. $\w\in A_n$.
 
 Now, for $\w\in A_n$, recall that the coefficients of $([R\,](\w))^{-1}$, namely $r_{i,j}(\w)$, are given by the formula $r_{i,j} (\w)= (-1)^{i+j} \det(([R\,](\w))_{ij})/\det([R\,](\w))$, where $[R\,](\w)_{ij}$ denotes the minor matrix obtained from $[R\,](\w)$ removing the $i$-th row and $j$-th column. Since these operations preserve measurability and $[R\,](\w)$ has measurable functions as entries, we conclude that $([R\,](\w))^{-1}$ is a matrix with measurable entries as well. Finally, $A_n\ni\w \mapsto \langle ([R\,](\w))^{-1} e_i, e_{j}\rangle$ is a measurable scalar function and so is $A_n \ni \w\mapsto (R(\w))^{-1}$, as we wanted to prove. 
  \end{proof}

\subsection{Measurable eigenvalues}\label{measurable-eigenvalues}

In this subsection we will prove the existence of measurable eigenvalues of $R(\w)$.
We begin with the case where $\dim J(\w)$ is constant and finite for a.e. $\w\in[0,1)^d$. 

First, we show that for an $n\times n$ matrix $M=M(\omega)$ of measurable functions there exists a choice of measurable eigenvalues.

 \begin{proposition}\label{prop:eigenvalues}
 If $M=M(\omega)$ is an $n\times n$ matrix of measurable functions defined on a measurable space $\Omega$, then there exist $n$ measurable functions $\lambda_j:\Omega\rightarrow \mathbb C$, $j=1,\dots,n$, such that $\lambda_1(\omega),\dots,\lambda_n(\omega)$ are the eigenvalues of $M(\omega)$ for a.e. $\w\in \Omega$, counted with multiplicity.
 \end{proposition} 
 \begin{proof}
 Let $M_n$ be the algebra of $n\times n$ complex matrices. By  \cite[Corollary 4]{A}, there exists a Borel measurable mapping $\mathcal{J} :M_n\to M_n$ such that for all  $A\in M_n,\mathcal{J}(A)$ is the Jordan canonical form of $A$. Since $M:\Omega\to M_n$ is measurable, then $\mathcal{J}\circ M:\Omega \to M_n$ is measurable. Now, as the diagonal entries of the  Jordan canonical form of a matrix are its eigenvalues, the result follows.
 \end{proof} 
As a consequence we have:

 \begin{theorem}\label{cor:eigenvalues-of-R}
Let $R(\omega):J(\omega)\rightarrow J(\omega)$ be a measurable range operator on a range function $J$ satisfying $\dim J(\w)=n$ a.e. $\w\in B$ where $B\subseteq [0,1)^d$ is measurable. Then, there exist $n$ measurable functions $\lambda_j:B\rightarrow \mathbb C$, $j=1,\dots,n$, such that $\lambda_1(\omega),\dots,\lambda_n(\omega)$ are the eigenvalues of $R(\omega)$ a.e. $\w\in B$, counted with multiplicity.
 \end{theorem}
 \begin{proof}
When $\dim J(\w) = n$ almost everywhere in a measurable set $B\subseteq [0,1)^d$, the range operator $R(\w)$ can be seen as an $n\times n$ matrix of measurable functions over $B$, which we denote as $[R\,](\w)$ (see Proposition \ref{matrix R}). Then, if $\nu(\w):J(\w)\to\mathbb C^n$ is defined as in \eqref{eq:cambio-de-base} we have that 
 $$R(\w)=\nu(\w)^{-1}[R\,](\w)\nu(\w),\quad\textrm{ for a.e. }\w\in B.$$
 Thus, if $\lambda:B\to\mathbb C$ is a measurable function, it holds that $\ker\left(R(\w)-\lambda(\w)\mathcal {I}_\omega\right)=\ker\left(\nu(\w)^{-1}([R\,](\w)-\lambda(\w)\mathcal {I}_\omega)\nu(\w)\right)$ for a.e. $\w\in B.$
 As a consequence, measurable eigenvalues of $R(\w)$ and $[R\,](\w)$ agree, and Proposition \ref{prop:eigenvalues} gives what we wanted. 
 \end{proof}
 In particular, the theorem holds for $B=A_n$, where $A_n$ is defined in Lemma \ref{An}.
 
 For the more general case when  $\dim J(\w)<\infty$ a.e. $\w\in [0,1)^d$
 we can construct $L^\infty$-functions such that when restricted to some measurable sets, they are eigenvalues of $R(\w)$. 
 
 Given a bounded measurable range operator $R$ we will denote 
 \begin{equation}\label{eq:spectrum of R}
 \Sigma(\w):=\{\lambda\in\mathbb C\,:\,\lambda\text{ is an eigenvalue of }R(\w)\}.
 \end{equation}

	\begin{theorem}\label{thm:autovalores-full}
 	Let $R(\w):J(\w)\rightarrow J(\w)$ be a bounded measurable range operator on a range function $J$ satisfying $\dim J(\w)<\infty$ for a.e. $\w\in [0,1)^d$. Then, there exist  functions $\lambda_j\in L^\infty([0,1)^d) $, $j\in\N$, such that
 	\begin{enumerate}[\rm (i)]
 		\item  $\lambda_j(\w)\neq\lambda_{j'}(\w)$ for $j\neq j'$ and for a.e. $\w\in[0,1)^d$,
 		\item if $\{A_n\}_{n\in\N_0}$ are the sets of Lemma \ref{An} and $A_{n,i}:= \left\{\w\in A_n\,:\, \#\Sigma(\w) = i\right\}$, then $\Sigma(\w)=\{\lambda_1(\w),\dots,\lambda_i(\omega)\}$ for a.e. $\w\in A_{n,i}$ and for every $i\leq n$, $i,n \in \N$.
 	\end{enumerate} 
 \end{theorem}

\begin{proof}
	Let $\{A_n\}_{n\in\N_0}$ be the disjoint measurable sets given by Lemma \ref{An}. We have that $\sigma(V)=\bigcup_{n\in\N}A_n$.
	For every $n\in\N$, by Theorem \ref{cor:eigenvalues-of-R}, there exist $n$ measurable functions defined on $A_n$, namely 
	$$\lambda^n_1,\dots,\lambda^n_n:A_n \to \mathbb C,$$
	such that at a.e. $\w\in A_n$ these are the eigenvalues of $R(\w)$, counted with multiplicity.
	
	Now, for a fixed $n\in\N$ and for every $1 \leq i\leq n$  the sets $A_{n,i}$ are  
	\begin{equation*}\label{eq:sigmaji}
		A_{n,i}= \left\{\w\in A_n\,:\, \#\{\lambda^n_1(\w),\dots,\lambda^n_n(\w)\} = i\right\}.
	\end{equation*}
	Hence, these are also disjoint measurable sets (possibly of measure zero) such that $A_n=\bigcup_{i=1}^n A_{n,i}$, for every $n\in\N$. 
	On each of these sets $A_{n,i}$ we have $i$ measurable functions
	$$\lambda_1^{n,i},\dots,\lambda_i^{n,i}:A_{n,i}\to\mathbb C,$$
	such that $\lambda_j^{n,i}(\w)$ is an eigenvalue of $R(\w)$ and $\lambda_j^{n,i}(\w) \neq \lambda_{j'}^{n,i}(\w)$ when $j\neq j'$ for a.e. $\w\in A_{n,i}$. The measurability of  $\lambda_j^{n,i}$
	deserves a moment of thought: partitioning 
	$A_{n,i}$  into all possibilities on which $i$ eigenvalues from $\{\lambda^n_1(\w),\dots,\lambda^n_n(\w)\}$ are different and pasting them properly we get measurability. 
	
	Moreover, assume that $R$ is uniformly bounded from above by a constant $K>0$, then we have that $|\lambda_{j}^{n,i}(\w)|\leq K$ for a.e. $\w\in A_{n,i}$, and for every $j\leq i \leq n$ and $n\in\N$. We now proceed to paste these functions in the following way:	
	
	For every $j\in\N$, let $\lambda_j:[0,1)^d\to \mathbb C$ be defined as
	$$\lambda_j(\w) := \begin{cases}
	\lambda_j^{n,i}(\w), &\text{ when }\w\in A_{n,i}, \text{ for } \, n\geq i \geq j,\\
	K + j &\text{ otherwise.}
	\end{cases}$$
	These are measurable functions satisfying that $\lambda_j(\w) \neq \lambda_{j'}(\w)$ when $j\neq j'$ for a.e. $\w\in [0,1)^d$
	and  $\lambda_j\in L^\infty([0,1)^d)$ for $j\in\N$. 
	
	Finally, note that for $j\in\N$,
	$\ker(R(\w)-\lambda_{j}(\w)\mathcal {I}_\omega)=\ker(R(\w)-\lambda_{j}^{n,i}(\w)\mathcal {I}_\omega)$ for a.e.  $\w\in A_{n,i}$ and for every $n\geq i \geq j$. Therefore, $\lambda_{j}(\w)$ is an eigenvalue of $R(\w)$ for a.e. $\w\in A_{n,i}$ and for every $n\geq i \geq j$. Otherwise, $\ker(R(\w)-\lambda_{j}(\w)\mathcal {I}_\omega)=\{0\}$ since $\lambda_j(\w)=K+j$ is not an eigenvalue of $R(\w)$.
\end{proof}

\begin{remark}\label{inclusion of spec}
\

(i) Assume that $J$ is the range function of a shift-invariant space $V$ which is  finitely generated. Then, for only finitely many values of $n$, the sets $A_n$ have positive measure. In particular, $|A_n|=0$ for all $ n> \mathcal L(V)$. In that case, if we discard the functions $\lambda_j$ such that, for a.e. $\w\in[0,1)^d$, $\lambda_j(\w)$ is not an eigenvalue of $R(\w)$, the above procedure will generate at most $\mathcal L(V)$ functions. 

 More precisely,  for $i\in\N$, we define \begin{equation}\label{Bi}B_i:=\bigcup_{n=i}^{\infty} A_{n,i},\quad \text{ and }\quad  \mathfrak{g}:=\max\{i\in\N\,:\,|B_i|>0\}.\end{equation} Observe that, in fact,  
$B_i$ is the set of all the $\w\in\sigma(V)$ for which $R(\w)$ has exactly $i$ different eigenvalues. Hence, it is easy to see that
\begin{equation}\label{r esssup}
	\mathfrak{g}=\underset{\omega\in \sigma(V)}{\text{\rm ess sup }} \, \#\Sigma(\w),
\end{equation}	
where $\Sigma(\w)$ is as in \eqref{eq:spectrum of R}.
Then, we have that $\mathfrak{g}$ is finite, with $\mathfrak{g}\leq \mathcal L (V)$ and the procedure will generate exactly $\mathfrak{g}$ functions.

(ii) For $j\in \N$, define the sets 
\begin{equation*}\label{Cs} C_j:=\bigcup_{i=j}^{\infty} B_i,\end{equation*} 
where $B_i$ are the ones defined in \eqref{Bi}. Notice that $C_{j+1}\subseteq C_{j}\subseteq \sigma(V)$ for every $j\in \N$. Each $C_j$ is the set of all the $\w\in\sigma(V)$ for which $R(\w)$ has at least $j$ different eigenvalues. Furthermore, from the construction done in Theorem \ref{thm:autovalores-full}, we also observe that $$C_j=\left\{\w\in[0,1)^d\,:\, \ker(R(\w)-\lambda_{j}(\w)\mathcal {I}_\omega)\neq\{0\}\right\},$$
for every $j\in\N$ and $|C_j|=0$ for every $j>\mathfrak{g}$.

(iii) Later on, we will assume that  $J$ satisfies  $\dim J(\w)\leq\ell$ for a.e. $\w\in [0,1)^d$, for some $0\leq\ell<\infty$. In that case, the functions constructed in Theorem \ref{thm:autovalores-full} will be used to define $s$-eigenvalues of the shift-preserving operator associated to $R$. 

(iv) The functions $\{\lambda_j\}_{j\in\N}$ of Theorem \ref{thm:autovalores-full} are, clearly,  not unique.

\end{remark}

 \section{$s$-diagonalization}\label{section-s-todo}

In this section, we introduce a new property for a bounded shift-preserving operator $L$ acting on a finitely generated shift-invariant space $V$, which we call {\it $s$-diagonalization}. In order to establish conditions on $L$ for being $s$-diagonalizable, we exploit the finite dimensional structure of its range function and its range operator: for almost each $\omega\in [0,1)^d,$ we have that $R(\w)$ is a linear transformation acting on a vector space of finite dimension. Thus, for instance, if $L$ is a normal operator, $R(\w)$ is normal and then diagonalizable  for a.e. $\w\in[0,1)^d$.  

 \subsection{Definition and properties of s-diagonalizations}

 \begin{definition}\label{s-diag}
 	Let $V$ be a finitely generated shift-invariant space and  $L:V\rightarrow V$ a bounded shift-preserving operator. We say that $L$ is {\it $s$-diagonalizable} if there exists a finite number of sequences of bounded spectrum $a_1,\dots,a_m$  such that $\Lambda_{a_1}, \dots, \Lambda_{a_m}$ are $s$-eigenvalues of $L$  and
 	\begin{equation}\label{direct sum V}
		V = V_{a_1}\oplus \dots \oplus V_{a_m},
	\end{equation}
 where $\Lambda_{a_j}$ and $V_{a_j}$ are the ones defined in Definition \ref{def: s-eigenvalues} for $j=1,\dots, m$. In that case, we say that $(a_1, \dots, a_m)$ is an {\it $s$-diagonalization} of $L$.

 \end{definition}
 
 \begin{remark}\label{s-eigenvalues not unique}
 \ 
 
 (i) Notice that if $(a_1, \dots, a_m)$ is an {\it $s$-diagonalization} of $L$, then each $f\in V$ can be written in a unique way as $f= f_1+\dots +f_m$ with $f_j\in V_{a_j}$.
 On the other hand, we have $Lf_j= \Lambda_{a_j} f_j,$  for $\; j=1,...,m$.
 If we denote by $Q_{V_{a_j} }:V \rightarrow V$ the oblique projection onto $V_{a_j}$ respect to this decomposition, i.e. $Q_{V_{a_j}}f =f_j$ we can write:
 $$L=\sum_{j=1}^m \Lambda_{a_j} Q_{V_{a_j}}.$$
 
(ii) The decomposition in \eqref{direct sum V} is not unique. Furthermore, the number $m$ of 
 $s$-eigenspaces of $L$ for which $V$ admits a decomposition like in \eqref{direct sum V} could be arbitrarily large. Indeed, take for example the first $s$-eigenvalue $\Lambda_{a_1}$, then, by Proposition \ref{prop-eigen}, $\lambda_{a_1}(\w)$ is an eigenvalue for $R(\w)$ for a.e. $\w\in\sigma(V_{a_1})$. Now, split $\sigma(V_{a_1})=A\cup B$ where $A\cap B=\emptyset$ and both sets are of positive measure. Then, define $\lambda_A(\w):=\lambda_{a_1}(\w)\chi_{A}(\w)$ and $\lambda_B(\w):=\lambda_{a_1}(\w)\chi_{B}(\w)$. These are measurable functions which are eigenvalues of $R(\w)$ over a positive measure set. By Lemma \ref{eigenvalue of R}, it is possible to construct sequences $a_A$ and $a_B$ of bounded spectrum such that $\lambda_A(\w)=\hat{a}_A(\w)$ for a.e. $\w\in A$, $\lambda_B(\w)=\hat{a}_B(\w)$ for a.e. $\w\in B$, $\Lambda_{a_A},\Lambda_{a_B}$ are $s$-eigenvalues of $L$ and $V_{a_1}=V_{a_A}\oplus V_{a_B}$. Finally, we obtain the decomposition $V=V_{a_A}\oplus V_{a_B}\oplus V_{a_2}\oplus \dots \oplus V_{a_m}$.
 \end{remark}

The next proposition shows that when $L$ is invertible and $s$-diagonalizable, $L^{-1}$ must be $s$-diagonalizable as well.

\begin{proposition}\label{prop:inverso}
	 Let $V$ be a finitely generated shift-invariant space and  $L:V\rightarrow V$ a bounded shift-preserving operator which is invertible. If $L$ is $s$-diagonalizable then $L^{-1}$ is $s$-diagonalizable. Furthermore, if $(a_1, \dots, a_m)$ is an $s$-diagonalization of $L$, then $(b_1,\dots,b_m)$ is an $s$-diagonalization of $L^{-1}$ with
	 $\hat{b}_j(\w)=(\hat{a}_j(\w))^{-1}$  for a.e. $\w\in[0,1)^d$ and $V_{b_j}=\ker(L^{-1} - \Lambda_{b_j})=V_{a_j}$ for $j=1,\dots,m$.
	 \end{proposition} 
	 \begin{proof}
	 Fix $j\in\{1,\dots, m\}$. Since $L$ is invertible and $L\big |_{V_{a_j}}=\Lambda_{a_j}:V_{a_j}\to V_{a_j}$, we see that $\Lambda_{a_j}$ is invertible. Let $R_{a_j}$ be the range operator corresponding to $\Lambda_{a_j}$, this is $R_{a_j}(\w) = \hat{a}_j(\w)\mathcal {I}_\omega$ for a.e. $\w\in[0,1)^d$. By Theorem 4.6 in \cite{B}, 
we get that $R_{a_j}(\w)$ is uniformly bounded from below by a constant $A>0$ for a.e. $\w\in[0,1)^d$, which implies that $|\hat{a}_j(\w)|\geq A$ for a.e. $\w\in[0,1)^d$. Also, since $a_j$ is of bounded spectrum, there exists a constant $B>0$ such that $A\leq |\hat{a}_j(\w)|\leq B$ for a.e. $\w\in[0,1)^d$.

Define $b_j\in\ell^2(\Z^d)$ by its Fourier transform as $\hat{b}_j(\w):=(\hat{a}_j(\w))^{-1}$ for a.e. $\w\in [0,1)^d$. Then, we have that $B^{-1} \leq |\hat{b}_j(\w)| \leq A^{-1}$ for a.e. $\w\in[0,1)^d$ and so $b_j$ is of bounded spectrum. Let us see that $\Lambda_{b_j}$ is an $s$-eigenvalue of $L^{-1}$. Let $f\in V_{a_j}$, then $\cT(Lf) (\w)= \hat{a}_j(\w)\cT f(\w)$ for a.e. $\w\in [0,1)^d$. Now, the function $g:=Lf$ is in $V_{a_j}$ and we have that $\cT g (\w) = \hat{a}_j(\w) \cT (L^{-1} g)(\w)$ for a.e. $\w\in [0,1)^d$, which is the same that $\hat{b}_j(\w) \cT g(\w) = \cT (L^{-1}g)(\w)$ for a.e. $\w\in[0,1)^d$. This proves that for every $g\in V_{a_j}$, $L^{-1}g= \Lambda_{b_j} g$ and thus $V_{a_j}\subseteq \ker (L^{-1} - \Lambda_{b_j} )$. A similar argument shows that in fact $V_{a_j} = \ker (L^{-1} - \Lambda_{b_j} )$.
	 \end{proof}
 
One of our main goals is to explore the relation between the notion of $s$-diagonali\-zation for a shift-preserving operator $L$ and the diagonalization of the fibers of its range operator. In this direction, we have the following result.
 
  \begin{theorem}\label{s-diagonalizable then diagonalizable}
	 Let $V$ be a finitely generated shift-invariant space and  $L:V\rightarrow V$ a bounded shift-preserving operator with range operator $R$. If $L$ is $s$-diagonalizable, then $R(\omega)$ is diagonalizable for a.e. $\omega\in\sigma(V)$.
	 
	  \end{theorem}
 
 \begin{proof}
 Suppose that $(a_1, \dots, a_m)$ is an $s$-diagonalization of $L$. Then, by Proposition \ref{prop-eigen}, for every $j=1,\dots,m$, $\lambda_{a_j}(\omega)=\hat{a}_j(\w)$ is an eigenvalue of $R(\w)$ for a.e. $\w\in \sigma(V_{a_j})$ and $J_{V_{a_j}}(\omega)= \ker(R(\w)-\lambda_{a_j}(\w)\mathcal {I}_\omega)$ is its associated eigenspace. We will see that $J(\w) = J_{V_{a_1}}(\w)\oplus \dots \oplus J_{V_{a_m}}(\w)$. 

First, notice that since $V_{a_j} \subseteq V$, we have that $J_{V_{a_j}}(\w) \subseteq J(\w)$ for every $j=1,\dots,m$, and hence $J_{V_{a_1}}(\w)+\dots+ J_{V_{a_m}}(\w) \subseteq J(\w)$. For the other inclusion, let $f\in V$, then $f = f_1 +\dots + f_m$, where $f_j \in V_{a_j}$ for $j=1,\dots, m$. Thus, $\mathcal T f(\omega) = \mathcal T f_1(\omega) + \dots + \mathcal T f_m(\omega)$ for a.e. $\omega \in [0,1)^d$, which implies that 
	$$J(\omega) = \overline{\rm span} \left\{\mathcal T \varphi(\omega)\,:\, \varphi\in \Phi \right\} \subseteq \overline{J_{V_{a_1}}(\omega)+\dots+J_{V_{a_m}}(\omega)},$$
	where $\Phi\subset L^2(\R^d)$ is such that $V=S(\Phi)$.
Since $J(\w)$ has finite dimension for a.e. $\w\in[0,1)^d$, the sum is closed.
	By (iii) of Lemma \ref{range-properties}, the sum is direct.
 \end{proof}

 \begin{remark}\label{rem:of Theo L diagonalizable}
	\
	
(i) Note that in the above theorem we have proved that if $J$ is the range function associated  to $V$ and  if $(a_1, \dots, a_m)$ is an $s$-diagonalization of $L$, then $J(\w) = J_{V_{a_1}}(\w)\oplus \dots \oplus J_{V_{a_m}}(\w)$, where   $J_{V_{a_j}}(\w)$ is the eigenspace associated to the eigenvalue $\lambda_{a_j}(\w)=\hat{a}_j(\w)$ of $R(\w)$ for a.e. $\w\in\sigma(V_{a_j})$ for $j=1,\dots,m$. From this, we can deduce that for a.e. $\w\in\sigma(V)$, the number of eigenvalues of $R(\w)$ is at most $m$.

(ii) Suppose $\mathcal{L}(V) = \ell$. In Remark \ref{s-eigenvalues not unique} we showed that the number of $s$-eigenspaces     for an $s$-diagonalization of  $L$ can be arbitrarily large. By Proposition \ref{different eigenvalues} and by the fact that $\dim J(\w)\leq \ell$ for a.e. $\w\in[0,1)^d$ we obtain that, for a.e. $\w\in\sigma(V)$, at most $\ell$ of the eigenspaces $J_{V_{a_j}}(\w)$ will not be the zero subspace.

\end{remark}

Let $V$ be a finitely generated shift-invariant space, $L:V\rightarrow V$ a bounded shift-preserving operator and $R$ its associated range operator. We will denote for $\w\in\sigma(V)$,
\begin{equation}\label{eq:K} 
k(\w):=\#\Sigma(\w),
\end{equation}
where $\Sigma(\w)$ is as in \eqref{eq:spectrum of R}.

Observe that, as seen in (i) of Remark \ref{rem:of Theo L diagonalizable}, given any $s$-diagonalization $(a_1,\dots,\allowbreak a_m)$ of $L$, we have that $k(\w)\leq m$ for a.e. $\w\in \sigma(V)$.  Moreover, we saw that if $\mathfrak{g}$ is as in \eqref{Bi}, then $\mathfrak{g} = \underset{\omega\in \sigma(V)}{\text{\rm ess sup }} k(\w)$ and therefore, $\mathfrak{g}\leq m$.

The next lemma shows an interesting relation between the spectrum of the $s$-eigenspaces of an $s$-diagonalization and the function $k(\w)$.
\begin{lemma}\label{prop:h}
 Let $V$ be a finitely generated shift-invariant space and $L:V\rightarrow V$ a bounded shift-preserving operator which is $s$-diagonalizable. Let $(a_1, \dots, a_m)$ be an $s$-diagonalization of $L$. Define for $\w\in\sigma(V)$,
 \begin{equation}\label{eq:function-h}
 h(\w):=\sum_{j=1}^{m} \chi_{\sigma(V_{a_j})}(\w),
 \end{equation}
 and let $k(\w)$ be defined as in \eqref{eq:K} for $\w\in\sigma(V)$. Then $h(\w)=k(\w)$ for a.e. $\w\in\sigma(V)$.
 \end{lemma}

\begin{proof}
  The function $h(\w)$ is measurable and we have that $h(\w)\in\{1,\dots,m\}$ for a.e. $\w\in\sigma(V)$.
Let $s\in\{1,\dots,m\}$ and define $\mathcal P_s:=\left\{ P\subseteq \{1,\dots,m\}\,:\, \#P =s\right\}$.
We have that $h^{-1}(s)=\bigcup_{P\in\mathcal P_s}B_P$ where $B_P:=\{\w\in\sigma(V)\,:\,\w\in \sigma(V_{a_j})\,\text{ if and only if }\allowbreak j\in P\}$. 

Given a set $P=\{j_1,\dots,j_s\}\in\mathcal P_s$ we see, by Proposition \ref{prop-eigen} and Theorem \ref{s-diagonalizable then diagonalizable}, that for a.e. $\w\in B_P$, $\hat{a}_{j_1}(\w),\dots,\hat{a}_{j_s}(\w)$ are eigenvalues of $R(\w)$ and
$$J(\w)=J_{V_{a_{j_1}}}(\w)\oplus\dots\oplus J_{V_{a_{j_s}}}(\w)\oplus \bigoplus_{1\leq j\leq m,\,j\notin P}\{0\}.$$
This shows that $k(\w)=s$ for a.e. $\w\in B_P$. 

Since this holds for every $P\in \mathcal P_s$ and for every $s\in\{1,\dots,m\}$, we can deduce that $h(\w)=k(\w)$ for a.e. $\w\in\sigma(V)$.
\end {proof}

Given an $s$-diagonalization $(a_1,\dots,a_m)$ of $L$, the collection of the spectrums $\{\sigma(V_{a_j})\}_{j=1}^m$ is always a covering of $\sigma(V)$. We will prove in Theorem \ref{converse with angle} that every $s$-diagonalizable shift-preserving operator $L$ acting on $V$ always has
an $s$-diagonalization $(a_1,\dots,a_m)$ whose spectrums satisfy
 $\sigma(V_{a_{j+1}})\subseteq \sigma(V_{a_j})$ for every $j=1,\dots,m$. In the following proposition we show that when this inclusion holds, the number of  $s$-eigenspaces in the asociated decomposition and its spectrums are uniquely determined.
Later we will see (Theorem \ref{minimal}) that actually a decomposition that satisfies the above inclusions has a minimal posible number of $s$-eigenspaces.
\begin{proposition}\label{prop:h2}
 Let $V$ be a finitely generated shift-invariant space and $L:V\rightarrow V$ a bounded shift-preserving operator which is $s$-diagonalizable. Let $(a_1, \dots, a_m)$ be an $s$-diagonalization of $L$.  If $\sigma(V_{a_{j+1}})\subseteq \sigma(V_{a_j})$ for every $j=1,\dots,m$, then the spectrums are uniquely determined by
 $$\sigma(V_{a_j})=\{\w\in\sigma(V)\,:\,k(\w)\geq j\}$$  
 for $j=1,\dots,m$ and
 $m=\mathfrak{g}$, where $\mathfrak{g}$ is as in \eqref{Bi}.
\end{proposition}

\begin{proof}
 Since $\sigma(V_{a_{j+1}})\subseteq \sigma(V_{a_j})$ for every $j=1,\dots, m$, it is easy to see that $\sigma(V_{a_j})=\{\w\in\sigma(V)\,:\,h(\w)\geq j\}$ where $h$ is the function defined in \eqref{eq:function-h}. Then, by Lemma \ref{prop:h}, we immediately  obtain that $\sigma(V_{a_j})=\{\w\in\sigma(V)\,:\,k(\w)\geq j\}$ for every $j=1,\dots,m$. 

We know that  $\mathfrak{g}\leq m$. Suppose that $\mathfrak{g}< m$. Then, we have that $|\sigma(V_{a_j})|=0$ for every $\mathfrak{g}<j\leq m$. This implies that $V_{a_j}=\{0\}$ for every $\mathfrak{g}<j\leq m$, which is not possible. Thus, $m=\mathfrak{g}$.
\end{proof}

\subsection{Range operators with diagonalizable fibers}
 The converse of Theorem \ref{s-diagonalizable then diagonalizable} requieres more work. If we assume that for  almost each $\omega$ the range operator $R(\omega)$ is diagonalizable,
 then we can obtain a decomposition  of $J(\omega)$ into a finite number of measurable eigenspaces as we show in the theorem below. However, in order that this decomposition corresponds to an $s$-diagonalization of the shift-preserving operator $L$ in $V$, we need to add an extra uniformity hypothesis which we will explain in the next subsection. 
  
 \begin{theorem}\label{converse-no angle}
Let $V$ be a finitely generated shift-invariant space with associated range function $J$, $L:V\rightarrow V$ a shift-preserving operator with corresponding range operator $R$ and $\mathfrak{g}$ as in \eqref{Bi}. If $R(\w)$ is diagonalizable for a.e. $\w\in\sigma(V)$,  then there exist sequences of bounded spectrum $a_1,\dots,a_{\mathfrak{g}}$ such that the measurable range functions $J_{a_j}(\w):=\ker(R(\w)-{\hat{a}_j}(\w)\mathcal {I}_\omega)$ for $j=1,\dots,\mathfrak{g}$ satisfy  the following statements:
\begin{enumerate}[\rm(i)]
\item $J(\w)=J_{a_1}(\w)\oplus \dots \oplus J_{a_{\mathfrak{g}}}(\w)$ for a.e. $\w\in [0,1)^d$.
\item The sets $C_j=\left\{\w\in\sigma(V)\,:\, J_{a_j}(\w)\neq \{0\}\right\}$ have positive measure for every $j=1,\dots,{\mathfrak{g}}$
and $C_{j+1}\subseteq C_j$ for every $j=1,\dots,{\mathfrak{g}}-1$.
\end{enumerate}
In particular, $\Lambda_{a_1}, \dots, \Lambda_{a_{\mathfrak{g}}}$ are $s$-eigenvalues of $L$.  
\end{theorem}

\begin{proof}
Assume that $\mathcal L(V)=\ell$. Then,  by Theorem \ref{thm:autovalores-full} and (i) of Remark  \ref{inclusion of spec}, there exist measurable functions $\lambda_1, \dots, \lambda_{\mathfrak{g}} \in L^\infty([0,1)^d)$ based on the partition  $\{A_{n,i}\}_{1\leq i \leq n\leq \ell}$ of $\sigma(V)$ given in Theorem \ref{thm:autovalores-full}.

Then, we have the decomposition
	\begin{equation}\label{decomp of J w lambda_j}
	J(\w) = \bigoplus_{j=1}^{\mathfrak{g}} \ker(R(\w)-\lambda_{j}(\w)\mathcal {I}_\omega),
	\end{equation}
	for a.e. $\w\in [0,1)^d$. Indeed, for a.e. $\w\notin \sigma(V)$ we have that $J(\w)=\{0\}$ and $\ker(R(\w)-\lambda_{j}(\w)\mathcal {I}_\omega)=\{0\}$ for $j=1,\dots, {\mathfrak{g}}$. Whereas, for a.e. $\w\in A_{n,i}$, for some $n,i$, we have that 
	$$\bigoplus_{j=1}^{\mathfrak{g}}\ker(R(\w)-\lambda_{j}(\w)\mathcal {I}_\omega) = \bigoplus_{j=1}^{i} \ker(R(\w)-\lambda_{j}^{n,i}(\w)\mathcal {I}_\omega) \oplus \bigoplus_{j=i}^{{\mathfrak{g}}} \{0\},$$
	which is equal to $J(\w)$ because $\{\lambda_{j}^{n,i}(\w)\}_{j=1}^i$ are the eigenvalues of  $R(\w)$ on $A_{n,i}$ and $R(\w)$ is diagonalizable for a.e. $\w\in [0,1)^d$.

Finally, for each $j=1,\dots,{\mathfrak{g}}$,  $\lambda_j\in L^\infty([0,1)^d)$, then, there exists a sequence $a_j\in \ell^2(\mathbb Z^d)$ of bounded spectrum such that, $\lambda_j(\w)=\hat{a}_j(\w)$   for a.e. $\w\in[0,1)^d$ and a measurable range function $J_{a_j}(\w)=\ker(R(\w)-\hat{a}_j(\w)\mathcal {I}_\omega)$. Since we have that \eqref{decomp of J w lambda_j} is true for a.e. $\w\in [0,1)^d$, these range functions decompose $J(\w)$ as we state in (i). Moreover, by item (ii) in Remark \ref{inclusion of spec}, we see that the sets $C_j$ have positive measure and $C_{j+1}\subseteq C_j$ for every $j=1,\dots, {\mathfrak{g}}-1$.
\end{proof}

In the conditions of Theorem \ref{converse-no angle}, 
 by Proposition \ref{prop-eigen}, each $J_{a_j}$ is the range function associated to the shift-invariant space $V_{a_j}$, i.e. $J_{a_j}=J_{V_{a_j}},$ and  we get the decomposition
$J(\w)=J_{V_{a_1}}(\w)\oplus \dots \oplus J_{V_{a_{\mathfrak{g}}}}(\w)$
 for a.e. $\w\in [0,1)^d$.
 At this point, one would like to conclude that $V=V_{a_1}\oplus\dots\oplus V_{a_{\mathfrak{g}}}$ and hence deduce that $L$ is $s$-diagonalizable. 
 For this to be true, a few considerations regarding the sum of infinite dimensional spaces must be taken. 
 Recall that, when the sum is not orthogonal, the sum of infinite dimensional closed subspaces is not necessarily closed. In particular, the sum of shift-invariant spaces is not closed in general, although it is invariant under integer translations. (See \cite{KKL2,AC} for examples of two shift-invariant spaces whose sum is not closed). A condition which allows us to tackle this problem is described in the next subsection. 
 
 \subsection{Angle between subspaces}
 
 We now turn to the definition of the angle between subspaces which provides a condition to determine  when  the sum of subspaces is closed (see \cite{D} and the references therein).
 
\begin{definition}
 	Let $\mathcal H$ be a Hilbert space. The {\it angle} of the $r$-tuple of closed subspaces $(M_1,\dots,M_r)$ in $\mathcal H$ is the angle in $[0,\pi/2]$ whose cosine is given by 
 	\begin{equation*}\label{c_b}
 	c_b(M_1,M_2,\dots,M_r):= \left\| P_{M_r}\dots P_{M_2}P_{M_1}P_{M_0^\perp} \right\|.
 	\end{equation*}
 \end{definition}
 where $M_0=\bigcap\limits_{j=1}^{r} M_j.$
 
 Notice that when $r=2$ this definition is equivalent to the definition of the Friedrichs angle between two spaces \cite{F}. The following lemma gives a characterization of when $c_b(M_1,\dots,M_r)<1$ and was proved in \cite[Theorem 3.7.4]{BBL}.
 
 \begin{lemma}\label{c_b<1} Let $M_1,\dots,M_r$ be closed subspaces of a Hilbert space $\mathcal H$. The subsequent statements are equivalent:
 	\begin{enumerate}[\rm (i)]
 		\item $c_b(M_1,\dots,M_r)<1,$
 		\item $M_1^{\perp} + \dots + M_r^{\perp}$ is closed.
 	\end{enumerate}
 \end{lemma}
 
 The cosine angle between two shift-invariant spaces was first considered in \cite{KKL1,KKL2}. In the next two results the shift-invariant spaces considered are not necessarily finitely generated. The following lemma gives a relation between the angle of a tuple of shift-invariant spaces and the angle of a tuple of their respective range functions. 
 
 \begin{lemma}\label{supess cb}
 	Let $V_1,\dots,V_r$ be shift-invariant spaces of $L^2(\R^d)$ and $J_{V_1},\dots,J_{V_r}$ their respective range functions. Then
 	\begin{equation}\label{c_b supess}
 	c_b(V_1,\dots,V_r) = \underset{\omega\in [0,1)^d}{\text{ess sup }} c_b\left( J_{V_1}(\w),\dots,J_{V_r}(\w)\right).
 	\end{equation}
 \end{lemma}
 
 This lemma was proved in \cite{AC} for the case of $r=2$. The proof is easy to extend to our context, as we show here.
 
 \begin{proof}
 	Let us call $U := \left(\bigcap\limits_{j=1}^{r} V_j\right)^{\perp}$. By virtue of (iii) in Lemma \ref{range-properties}, $U$ is a shift-invariant space and $$J_U(\w) = \left(\bigcap\limits_{j=1}^{r} J_{V_j}(\w)\right)^{\perp},$$
 	for a.e. $\w\in[0,1)^d$. 
 	
 	Since orthogonal projections onto shift-invariant spaces are shift-preserving operators (see Example \ref{ex:projection}), we have that $P_{V_r}\dots P_{V_2}P_{V_1}P_{U}$
 	is shift preserving and its range operator at $\w$ is 
 	$$P_{J_{V_r}(\w)}\dots P_{J_{V_2}(\w)}P_{J_{V_1}(\w)}P_{J_{U}(\w)}.$$
 	Hence, by \eqref{norm of R and L}
 	$$\underset{\omega\in [0,1)^d}{\esssup} \left\| P_{J_{V_r}(\w)}\dots P_{J_{V_2}(\w)}P_{J_{V_1}(\w)}P_{J_{U}(\w)} \right\| = \left\| P_{V_r}\dots P_{V_2}P_{V_1}P_{U}\right\|,$$
 	and consequently (\ref{c_b supess}) holds.
 \end{proof}
 
 As a consequence, we have: 
 
 \begin{proposition}\label{equivalence sum of sis}
 	Let $V_1,\dots,V_r$ be shift-invariant spaces of $L^2(\R^d)$ and $J_{V_1},\dots,J_{V_r}$ their respective range functions. The following statements are equivalent:
 	\begin{enumerate}[\rm(i)]
 		\item The space $U:=V_1+\dots+V_r$ is closed. In particular, it is a shift-invariant space of $L^2(\R^d)$. 
 		\item $c_b(V_1^{\perp},\dots,V_r^{\perp})<1$.
 		\item $\underset{\omega\in [0,1)^d}{\esssup\,}\, c_b\left( J_{V_1^{\perp}}(\w),\dots,J_{V_r^{\perp}}(\w)\right)<1$.
 		\item $\underset{\omega\in [0,1)^d}{\esssup}\,\, c_b\left( J_{V_1}(\w)^{\perp},\dots,J_{V_r}(\w)^{\perp}\right)<1$.
 	\end{enumerate}
 	Moreover, if {\rm (i)-(iv)} hold, then we have that $J_{U}(\w) = J_{V_1}(\w)+\dots+ J_{V_r}(\w)$, for a.e. $\w\in[0,1)^d$.
 \end{proposition}
 
 \begin{proof}
 	From Lemma \ref{c_b<1} we know that (i) and (ii) are equivalent. By Lemma \ref{supess cb} we see that (ii) is equivalent to (iii). Finally, the equivalence between (iii) and (iv) is immediate from item (i) in Lemma \ref{range-properties}.
 	
 	It remains to see that $J_{U}(\w) = J_{V_1}(\w)+\dots +J_{V_r}(\w)$. Analogously as in Theorem \ref{s-diagonalizable then diagonalizable}, we see that 
 	$J_{V_1}(\w)+\dots+ J_{V_r}(\w) \subseteq J_{U}(\w)$ and 
 	$J_{U}(\omega) \subseteq \overline{J_{V_1}(\omega)+\dots+J_{V_{r}}(\omega)}.$
 	These spaces are not necessarily finite dimensional but, given that (iv) holds, the latter sum is closed for almost every $\w$.
 \end{proof}
 
 \subsection{Sufficient conditions for $s$-diagonalization and minimality}
Using Proposition \ref{equivalence sum of sis}, we are able to give sufficient   conditions for a shift-preserving operator to be  $s$-diagonalizable. 

As we already discussed, in order to deduce that $L$ is $s$-diagonalizable from the hypothesis that $R(\w)$ is diagonalizable for a.e. $\w,$ an uniformity condition on the angle must be required.

For this, let $R$ be a measurable range operator on $J$ such that $R(\w)$ is diagonalizable almost everywhere and let $Z\subseteq [0,1)^d$ be the exceptional zero measure set. Then, for $\w\in[0,1)^d\setminus Z$ and if $k(\w)$ is as in \eqref{eq:K}, we know that 
$$J(\w)=E_{\mu_1(\w)}\oplus \dots \oplus E_{\mu_{k(\w)}(\w)},$$
where $\mu_1(\w),\dots, \mu_{k(\w)}(\w)$ are the different eigenvalues of $R(\w)$ and $E_{\mu_1(\w)}\dots E_{\mu_{k(\w)}(\w)}$ are their associated eigenspaces. Note that at this point we do not care about measurability. Then, we can define
\begin{equation}\label{eq:Cb}
C_b(\w):=c_b\left(E_{\mu_1(\w)}^{\perp},\dots, E_{\mu_{k(\w)}(\w)}^{\perp}\right).
\end{equation}
Suppose now that  $a_1,\dots,a_m\in\ell^2(\Z^d)$
are any sequences of bounded spectrum such that $\{J_{a_j}(\w):=\ker(R(\w)-{\hat{a}_j}(\w)\mathcal {I}_\omega)\}_{j=1}^m$ decomposes $J(\w)$ as 
$$J(\w)=J_{a_1}(\w)\oplus \dots \oplus J_{a_m}(\w)$$
for a.e. $\w\in [0,1)^d$. The existence of such sequences is guaranteed by Theorem \ref{converse-no angle}. Thus, for $\w\in [0,1)^d\setminus Z$ we have that 
$$\{E_{\mu_1(\w)}\dots E_{\mu_{k(\w)}(\w)}\}=\{J_{a_1}(\w), \dots, J_{a_m}(\w) \}\setminus\{0\}.$$
Therefore, $C_b(\w)=c_b\left( J_{a_1}(\w)^{\perp},\dots,J_{a_m}(\w)^{\perp}\right)$ almost everywhere and as a consequence $C_b$ is measurable.   

 \begin{theorem}\label{converse with angle}
 	Let $V$ be a finitely generated shift-invariant space with associated range function $J$ and $L:V\rightarrow V$ a shift-preserving operator with corresponding range operator $R$. 
	The following propositions are equivalent:
\begin{enumerate}[\rm (i)]
	\item\label{1} $L$ is $s$-diagonalizable.
	\item\label{2} $R(\w)$ is diagonalizable for a.e. $\w\in \sigma(V)$
 	and  $\underset{\omega\in [0,1)^d}{\text{\rm ess sup }} C_b(\w)<1$.
 \end{enumerate}
 
Furthermore, when \ref{1} or \ref{2} holds, there  exist sequences $a_1,\dots,a_{\mathfrak{g}}$ of bounded spectrum, such that $(a_1,\dots,a_{\mathfrak{g}})$ is an $s$-diagonalization of $L$ and $\sigma(V_{a_{j+1}})\subseteq\sigma(V_{a_j})$ for $j=1,\dots,\mathfrak{g}-1$, where $\mathfrak{g}$ is as in \eqref{Bi}.
	 \end{theorem}
 
 \begin{proof}
 Assume first that \ref{2} is satisfied and consider the sequences of bounded spectrum $a_1,\dots,a_{\mathfrak{g}}$ given by Theorem \ref{converse-no angle} and their respective  measurable range functions  
$J_{a_j}(\w):=\ker(R(\w)-{\hat{a}_j}(\w)\mathcal {I}_\omega)$ for $j=1,\dots,\mathfrak{g}$.

By Proposition \ref{prop-eigen}, we have that $V_{a_j}=\ker(L-\Lambda_{a_j})$ is the shift-invariant space associated to the range function $J_{a_j}$, i.e. $J_{a_j}=J_{V_{a_j}}$. From (ii) in Theorem \ref{converse-no angle}, we know that $\sigma(V_{a_j})$ has positive measure, $\Lambda_{a_j}$ is an $s$-eigenvalue of $L$ for every $j=1,\dots, \mathfrak{g}$ and  
$\sigma(V_{a_{j+1}})\subseteq\sigma(V_{a_j})$ for $j=1,\dots,\mathfrak{g}-1$.
	
Since $C_b(\w)=c_b\left( J_{V_{a_1}}(\w)^{\perp},\dots,J_{V_{a_{\mathfrak{g}}}}(\w)^{\perp}\right)$, by Proposition \ref{equivalence sum of sis}, we have that $V_{a_1}\oplus \dots \oplus V_{a_{\mathfrak{g}}}$ is a closed shift-invariant space which is contained in $V$ and their range functions coincide. By (ii) in Lemma \ref{range-properties}, this implies that $V = V_{a_1}\oplus \dots \oplus V_{a_{\mathfrak{g}}}.$ Hence, $L$ is $s$-diagonalizable and $(V,L,a_1,\dots,a_{\mathfrak{g}})$ is an $s$-diagonalization of $L$.

For the converse, Theorem \ref{s-diagonalizable then diagonalizable} takes care of the diagonalization of $R(\w)$ for a.e. $\w\in\sigma(V).$
On the other hand, assume that $(a_1, \dots, a_m)$ is an $s$-diagonalization of $L$. Since $V$ is closed, by Proposition \ref{equivalence sum of sis} we have that $c_b(V_{a_1}^{\perp},\dots,V_{a_m}^{\perp})<1$ and so, equivalently, $\underset{\omega\in [0,1)^d}{\text{\rm ess sup }} C_b(\w)<1$.
 	 \end{proof}

Now, we characterize the minimum number of components in an $s$-diagonali-zation of L.

 \begin{definition}\label{betaVL}
 	Given $V$  a finitely generated shift-invariant space and $L:V\rightarrow V$ a bounded shift-preserving operator which is $s$-diagonalizable, we define $ \beta(V,L)$ as the smallest natural number $m$ for which there exist $a_1,\dots,a_m$ sequences of bounded spectrum such that $(a_1, \dots, a_m)$ is an $s$-diagonalization of $L$. We will say that
 	an $s$-diagonalization is {\it minimal } if it has exactly $\beta(V,L)$ components.
 \end{definition}

 \begin{proposition}\label{minimal}
 	Let $V$ be a finitely generated shift-invariant space, $L:V\rightarrow V$ a bounded shift-preserving operator which is $s$-diagonalizable.Then, if $\mathfrak{g}$ is as in \eqref{Bi},	$$\beta(V,L)=\mathfrak{g}.$$
 \end{proposition}
 
 	\begin{proof}
 		Let $(a_{1},...,a_{m})$ be a minimal $s$-diagonalization of $L$, $m=\beta(V,L)$. Observe that (i) of Remark \ref{rem:of Theo L diagonalizable} asserts that the number of eigenvalues of $R(\w)$ is at most $\beta(V,L)$ for a.e. $\w \in \sigma(V)$. Hence, $\mathfrak{g} \leq \beta(V,L)$. 
 		
 		By Theorem \ref{converse with angle}, there always exists an $s$-diagonalization of $L$ whose number of $s$-eigenvalues is $\mathfrak{g}$, then $\beta(V,L)\leq \mathfrak{g}$.
 	\end{proof}

 \subsection{Normal shift-preserving  operators}
 
  When $L$ is normal, $R(\w)$ is normal for a.e. $\w\in[0,1)^d$ and so its eigenspaces are orthogonal. This fact allows us to avoid the angle condition because the orthogonal sum of closed subspaces is always closed. Below is given a generalization of the finite dimensional Spectral Theorem, for bounded normal shift-preserving operators. 
 
\begin{theorem}\label{thm:spectral}
	Let $V$ be a finitely generated shift-invariant space and  $L:V\rightarrow V$ a bounded shift-preserving operator. If $L$ is normal then it is $s$-diagonalizable and, if $(a_1, \dots, a_m)$ is an $s$-diagonalization of $L$, we have that
	\begin{equation}\label{spectral theorem L}
		L= \sum_{j=1}^{m} \Lambda_{a_j} P_{V_{a_j}},
	\end{equation}
where $P_{V_{a_j}}$ denotes the orthogonal projection of $V$ onto $V_{a_j}$ for $j=1, \dots, m$.
\end{theorem}

\begin{proof}
	
	If $L$ is normal, by Theorem \ref{adjoint L},  we know that for a.e. $\omega\in[0,1)^d$, $R(\omega)$ is a normal operator acting on a finite-dimensional space $J(\w)$. Thus,  for a.e. $\w\in[0,1)^d$, $R(\w)$ is diagonalizable  and its eigenspaces  are orthogonal.  
	
	Let $\mathfrak{g}$ be as in \eqref{Bi}. By Theorem \ref{converse-no angle}, there exist sequences of bounded spectrum $a_1,\dots,a_\mathfrak{g}$ and measurable range functions $J_{a_j}(\w)=\ker(R(\w)-\hat{a}_j\mathcal I_\w)$, for $j=1,\dots,\mathfrak{g}$, such that $J(\w)=J_{a_1}(\w)\poplus \dots\poplus J_{a_\mathfrak{g}}(\w)$ for a.e. $\w\in[0,1)^d$.
	
	By (ii) in Proposition \ref{prop-eigen}, the subspace $V_{a_j}=\ker(L-\Lambda_{a_j})$ is the shift-invariant space associated to the range function $J_{a_j}$, that is, $J_{a_j}=J_{V_{a_j}}$. In addition, the orthogonality between the eigenspaces $J_{V_{a_j}}(\w)$ for a.e. $\w\in [0,1)^d$, implies that the $s$-eigenspaces $V_{a_j}$ are orthogonal.  Thus, we have that the orthogonal sum $V_{a_1}\poplus\dots\poplus V_{a_\mathfrak{g}}$ is a closed shift-invariant space which is contained in $V$ and their range functions coincide. By (ii) in Lemma \ref{range-properties} we conclude that $V=V_{a_1}\poplus\dots \poplus V_{a_\mathfrak{g}}$ and so $L$ is $s$-diagonalizable. 
	
	Moreover, let $(a_1,\dots,a_m)$ be any $s$-diagonalization of $L$ (not necessarily the mi\-nimal $s$-diagonalization). As before, since the eigenspaces of $R(\w)$ are orthogonal for a.e. $\w\in[0,1)^d$, the $s$-eigenspaces $V_{a_j}$ are orthogonal and $V=V_{a_1}\poplus\dots\poplus V_{a_m}$. Hence, the decomposition of $L$ in (\ref{spectral theorem L}) holds forthwith.

\end{proof}

\begin{remark}
	\
	
	(i) Theorem \ref{thm:spectral} can be extended to the case of shift-invariant spaces whose range operator satisfies $\dim J(\w) <\infty$ for a.e. $\w\in[0,1)^d$ (not necessarily finitely generated). 
	
	Indeed, first observe that under that hypothesis, Theorem \ref{converse-no angle} can be extended giving possibly infinite sequences $\{a_j\}_{j}$ such that $J(\w)=\bigoplus_{j} \ker(R(\w)-\hat{a}_j(\w)\mathcal I_\w)$ for a.e. $\w\in[0,1)^d$.
	Then, in the same way as in the proof of Theorem \ref{thm:spectral}, we get
	$$V=\underset{j}{\bigpoplus} V_{a_j},$$
	where the orthogonal sum could possibly be infinite. This gives us that $L=\sum_{j} \Lambda_{a_j}P_{V_{a_j}}$, being the convergence of this series in the strong operator topology whenever it is infinite.

	(ii) For shift-invariant spaces such that the dimension of $J(\w)$ is infinite on a set of positive measure, it is not generally true that $R(\w)$ is diagonalizable a.e. $\w\in[0,1)^d$ (i.e. there exists an orthonormal basis of $J(\w)$ consisting of eigenvectors of $R(\w)$ for a.e. $\w\in[0,1)^d$). Thus, an extension of Theorem \ref{thm:spectral} is in general no longer possible. Even if $R(\w)$ were diagonalizable a.e. $\w\in[0,1)^d$, since Theorem \ref{thm:autovalores-full} strongly relies on the fact that $J(\w)$ is finite-dimensional a.e. $\w\in[0,1)^d$, different arguments are needed to obtain a measurable choice of eigenvalues of $R(\w)$. See \cite{BCCHM20} where some extensions are obtained. 
\end{remark}

\begin{example}
Let $V=S(\Phi)$ be a finitely generated shift-invariant space such that $E(\Phi)$ is a Bessel sequence in $L^2(\R^d)$. As we said in Example \ref{ex:frame}, the frame operator associated to $E(\Phi)$ is shift preserving. Since it is self-adjoint, by Theorem \ref{thm:spectral}, we have that it is $s$-diagonalizable.  
\end{example}

\begin{example}
In the setting of Example \ref{ex:convolution}, assume that $E(\Phi)$ is an orthonormal basis and let $L$ be the shift-preserving operator  defined in \eqref{eq:op-convolution}. Then, since the matrix form of $R(\w)$  written on the basis $\{\mathcal{T}\varphi_{1}(\omega),\dots,\mathcal{T}\varphi_{n}(\omega)\}$ is the diagonal matrix $[R\,](\w)=\textrm{diag}(\hat{a}_1(\w),\dots, \hat{a}_n(\w))$ for a.e. $\omega\in [0,1)^{d}$, we have that $R(\w)$ is normal a.e. $\omega\in [0,1)^{d}$. Observe that this is not the case when we have a Riesz basis instead of an orthonormal basis. As a consequence, $L$ is normal as well and by Theorem \ref{thm:spectral}, it is $s$-diagonalizable. Moreover, $\Lambda_{a_1}\dots\Lambda_{a_n}$ are the $s$-eigenvalues of $L$ and $V_{a_j}=S(\varphi_j)$ for every $j=1,\dots, n$. 

When $E(\Phi)$ is a Riesz basis, we still have that, for a.e. $\omega\in [0,1)^{d}$, $\hat{a}_1(\w),\dots, \hat{a}_n(\w)$ are the eigenvalues of $R(\w)$ with associated eigenvectors $\mathcal{T}\varphi_{1}(\omega),\dots,\mathcal{T}\varphi_{n}(\omega)$ and that
$$J(\w)=\textrm{span}\{\mathcal{T}\varphi_{1}(\w)\}\oplus\dots\oplus\textrm{span}\{\mathcal{T}\varphi_{n}(\w)\}.$$
However, in order to conclude that $L$ is $s$-diagonalizable we have to impose that $\underset{\omega\in [0,1)^d}{\text{\rm ess sup }} c_b\left( \textrm{span}\{\mathcal{T}\varphi_{1}(\w)\}^{\perp},\dots,\textrm{span}\{\mathcal{T}\varphi_{n}(\w)\}^{\perp}\right)<1$. 
\end{example}

\section{Acknowledgements}
We thank the referee for her/his careful reading of the manuscript and for the interesting remarks and suggestions that had really helped us to improve the paper.

\end{document}